\documentclass[11pt]{article}
\usepackage{enumerate}
\usepackage{amssymb,a4wide,latexsym,makeidx,epsfig}
\usepackage{amsthm}
\usepackage{dsfont}
\usepackage{amsmath}
\usepackage{lipsum}
\usepackage{enumerate}
\usepackage{mathrsfs}
\usepackage{xcolor}
\usepackage{tikz}
\usepackage{geometry}
\usepackage{appendix}
\usepackage{multirow}
\usepackage[colorlinks, linkcolor=black, anchorcolor=black, citecolor=black]{hyperref}
\usepackage{microtype}
\usepackage{colonequals}

\allowdisplaybreaks

\newtheorem{theorem}{Theorem}[section]
\newtheorem{remark}[theorem]{Remark}

\newtheorem{definition}[theorem]{Definition}
\newtheorem{lemma}[theorem]{Lemma}
\newtheorem{fact}[theorem]{Fact}

\newtheorem{proposition}[theorem]{Proposition}
\newtheorem{corollary}[theorem]{Corollary}
\newtheorem{problem}[theorem]{Problem}

\newtheorem{conjecture}[theorem]{Conjecture}
\newtheorem{claim}{Claim}[section]

\newcommand{\lf}{\left\lfloor}
\newcommand{\rf}{\right\rfloor}
\newcommand{\lc}{\left\lceil}
\newcommand{\rc}{\right\rceil}

\begin{document}
\textwidth 150mm \textheight 225mm
\title{Growth rates of the bipartite Erd\H{o}s-Gy\'{a}rf\'{a}s function
\thanks{Supported by the National Natural Science Foundation of China (Nos. 12271439 and 11871398), the Fundamental Research Funds for the Central Universities (WK0010000088),
and China Scholarship Council (No. 201906290174).}
}
\author{{Xihe Li$^{1,2,3}$, Hajo Broersma$^{2,}$\thanks{Corresponding author.}, Ligong Wang$^{1}$}\\
{\small $^{1}$ School of Mathematics and Statistics,}\\ {\small Northwestern Polytechnical University, Xi'an, Shaanxi 710129, PR China}\\
{\small $^{2}$ Faculty of Electrical Engineering, Mathematics and Computer Science,}\\ {\small University of Twente, P.O. Box 217, 7500 AE Enschede, The Netherlands}\\
{\small $^{3}$ School of Mathematical Sciences,}\\ {\small University of Science and Technology of China, Hefei, Anhui 230026, PR China}\\
{\small E-mail: lxhdhr@163.com; h.j.broersma@utwente.nl; lgwangmath@163.com}}
\date{}
\maketitle
\begin{center}
\begin{minipage}{120mm}
\vskip 0.3cm
\begin{center}
{\small {\bf Abstract}}
\end{center}
{\small Given two graphs $G, H$ and a positive integer $q$, an $(H,q)$-coloring of $G$ is an edge-coloring of $G$ such that every copy of $H$ in $G$ receives at least $q$ distinct colors. The bipartite Erd\H{o}s-Gy\'{a}rf\'{a}s function $r(K_{n,n}, K_{s,t}, q)$ is defined to be the minimum number of colors needed for $K_{n,n}$ to have a $(K_{s,t}, q)$-coloring. For balanced complete bipartite graphs $K_{p,p}$, the function $r(K_{n,n}, K_{p,p}, q)$ was studied systematically in [Axenovich, F\"{u}redi and Mubayi, {\it J. Combin. Theory Ser. B} {\bf 79} (2000), 66--86]. In this paper, we study the asymptotic behavior of this function for complete bipartite graphs $K_{s,t}$ that are not necessarily balanced. Our main results deal with thresholds and lower and upper bounds for the growth rate of this function, in particular for (sub)linear and (sub)quadratic growth. We also obtain new lower bounds for the balanced bipartite case, and improve several results given by Axenovich, F\"{u}redi and Mubayi. Our proof techniques are based on an extension to bipartite graphs of the recently developed Color Energy Method by Pohoata and Sheffer and its refinements, and a generalization of an old result due to Corr\'{a}di.
}
\end{minipage}
\end{center}

\section{Introduction}
\label{sec:ch-intro}

Our work is motivated by recent results on the Erd\H{o}s-Gy\'{a}rf\'{a}s function and its extension to bipartite graphs, as well as a recently developed proof technique called the Color Energy Method and its refinements. For our purpose of studying the behavior and thresholds for different growth rates of the analogue of the Erd\H{o}s-Gy\'{a}rf\'{a}s function for bipartite graphs, we extend this proof approach to bipartite graphs. Before we can present our results we need a short introduction and some terminology.

For two graphs $G, H$ and an integer $q$ with $2\leq q\leq |E(H)|$, an \emph{$(H, q)$-coloring} of $G$ is an edge-coloring of $G$ such that every copy of $H$ in $G$ receives at least $q$ distinct colors. Let $r(G, H, q)$ be the minimum number of colors that are needed for $G$ to have an $(H, q)$-coloring. In the case $G=K_n$ and $H=K_p$, $r(K_n, K_p, q)$ is usually written as $f(n,p,q)$ and known as the Erd\H{o}s-Gy\'{a}rf\'{a}s function. This function was first introduced by Erd\H{o}s and Shelah \cite{Erd1,Erd2} and studied in depth by Erd\H{o}s and Gy\'{a}rf\'{a}s \cite{ErGy} in 1997.

In the past two decades, quite a few articles on the topic of the Erd\H{o}s-Gy\'{a}rf\'{a}s problem have appeared. For the Erd\H{o}s-Gy\'{a}rf\'{a}s function $f(n,p,q)$, we refer the interested reader to \cite{Axe,Cam,CaHe,CFLS2,FoSu,Mub04,SaSe1,SaSe2}. For the bipartite Erd\H{o}s-Gy\'{a}rf\'{a}s function $r(K_{n,n}, K_{p,p}, q)$, we refer to \cite{AxFM,SaSe3}. In \cite{Kru}, Krueger studied the asymptotic behavior of $r(K_n, P_m, q)$. The Erd\H{o}s-Gy\'{a}rf\'{a}s function was also studied in the setting of hypergraphs, see \cite{CFLS1,Mub16}. In \cite{LiBW}, the authors investigated the Erd\H{o}s-Gy\'{a}rf\'{a}s function within the framework of Gallai-colorings. In \cite{FoPS}, Fox, Pach and Suk studied the semi-algebraic variant of the Erd\H{o}s-Gy\'{a}rf\'{a}s function. A chromatic number version of the Erd\H{o}s-Gy\'{a}rf\'{a}s function was considered in \cite{CFLS1,LeTr}. For more information on this topic, we refer the interested reader to \cite[Section~3.5.1]{CoFS} and \cite[Section~7]{MuSu}.

The results we will present next all deal with the growth rates of the bipartite Erd\H{o}s-Gy\'{a}rf\'{a}s function $r(K_{n,n}, K_{s,t}, q)$. In 2000, Axenovich, F\"{u}redi and Mubayi \cite{AxFM} studied $r(K_{n,n}, K_{p,p}, q)$ systematically. In particular, they determined various thresholds for different growth rates of $r(K_{n,n}, K_{p,p}, q)$; see Table~\ref{tab:rnn}.

\begin{table}[htb]
\begin{center}
\label{tab:rnn}
\begin{tabular}{l|l|l}
  \cline{1-3}
  $q$ & $r'\colonequals r(K_{n,n}, K_{p,p}, q)$ & Remark \\
  \cline{1-3}
  $p^2$ & $r'=n^2$ & \\
   $p^2-\lfloor p/2\rfloor+1$ & $r'=n^2-\lfloor p/2 \rfloor+1$ &  threshold for $r'=n^2-O(1)$\\
   $p^2-\lfloor p/2\rfloor$ & $r'\leq n^2-\lfloor n/2 \rfloor$ & \\
   $p^2-\lfloor (2p-1)/3\rfloor+1$ & $r'>n^2-2\lfloor (p-2)/3 \rfloor(n-1)$ &  threshold for $r'=n^2-O(n)$\\
   $p^2-\lfloor (2p-1)/3\rfloor$ & $r'< n^2-c_pn^{1+\varepsilon_p} $ & \\
   $p^2-p+2$ & $C_p(n^2-n)\leq r' \leq (1-C'_p)n^2$ &  threshold for quadratic $r'$\\
   $p^2-p+1$ & $n^{4/3}-2n^{2/3}+1\leq r' \leq c'_pn^{2-2/p}$ & \\
   $p^2-2p+3$ & $n/(2p-2)\leq r' \leq c_p''n$ &  threshold for linear $r'$\\
   $p^2-2p+2$ &  $r'\leq c'''_pn^{1-1/(2p-1)}$ & \\
   2 & $r'\geq(1+o(1))n^{1/p}$ & \\
  \cline{1-3}
\end{tabular}
\caption{Known results for $r(K_{n,n}, K_{p,p}, q)$ obtained in \cite{AxFM}.}
\end{center}
\vspace{-0.7cm}
\end{table}

The results in~\cite{AxFM} mainly concern the balanced complete bipartite graph $K_{p,p}$. In the present paper, we study the bipartite Erd\H{o}s-Gy\'{a}rf\'{a}s function $r(K_{n,n}, K_{s,t}, q)$, where $K_{s,t}$ is not necessarily balanced. This function generalizes the bipartite Ramsey number, since determining $r(K_{n,n}, K_{s,t}, 2)$ is equivalent to determining the multicolor bipartite Ramsey number of $K_{s,t}$.

We start with two easy observations that we give without proofs. We also note here that all proofs of our main results that follow are postponed to later sections. Most of our proofs require several auxiliary results and techniques which will be introduced in Section~\ref{sec:pre}, in particular our extension of the Color Energy Method to bipartite graphs that we present in Section~\ref{subsec:colorenergy}.

For any bipartite graphs $H, F$ and nonnegative integers $k, \ell$, we have the following two simple properties:
\begin{itemize}
\item[(i)] if $2\leq k\leq \ell\leq e(H)$, then $r(K_{n,n}, H, k)\leq r(K_{n,n}, H, \ell)$;
\item[(ii)] if $H\subseteq F$ and $e(H)- k\geq 2$, then $r(K_{n,n}, H, e(H)-k)\leq r(K_{n,n}, F, e(F)-k)$.
\end{itemize}
Here $e(H)=|E(H)|$ and $H\subseteq F$ denotes that $H$ is a subgraph of $F$.

For $t\geq s\geq 2$ and $2\leq q\leq st$, Axenovich, F\"{u}redi and Mubayi \cite{AxFM} established the following general upper bound (see \cite[Theorem~3.2]{AxFM}):
\begin{equation}\label{eq:UnbalanGener}
r(K_{n,n}, K_{s,t}, q)=O\left(n^{\frac{s+t-2}{st-q+1}}\right).
\end{equation}

As we will show in Section~\ref{sec:unbalanced}, the threshold for linear $r(K_{n,n}, K_{s,t}, q)$ is relatively easy to determine. The following result implies that for all $2\leq q\leq t$, the function $r(K_{n,n}, K_{1,t}, q)$ is linear in $n$, and if $t\geq s\geq 2$, then $st-s-t+3$ is the smallest $q$ such that $r(K_{n,n}, K_{s,t}, q)$ is linear in $n$.

\begin{theorem}\label{th:UnbalanceLin} Let $s, t, q$ be three positive integers.
\begin{itemize}
\item[{\rm (i)}] If $2\leq q\leq \frac{t+1}{2}$, then $\lc\frac{n}{(t-1)/(q-1)}\rc \leq r(K_{n,n}, K_{1,t}, q)\leq \lc\frac{n}{\lf (t-1)/(q-1)\rf }\rc$.
\item[{\rm (ii)}] If $\frac{t+2}{2}\leq q\leq t$, then $r(K_{n,n}, K_{1,t}, q)=n-t+q$.
\item[{\rm (iii)}] If $t\geq s\geq 2$ and $q=st-s-t+3$, then $r(K_{n,n}, K_{s,t}, q)=\Theta(n)$ and $r(K_{n,n}, K_{s,t}, q-1)= O\left(n^{1-\frac{1}{s+t-1}}\right)$.
\end{itemize}
\end{theorem}

Next, we consider the threshold for quadratic $r(K_{n,n}, K_{s,t}, q)$. The upper bound (\ref{eq:UnbalanGener}) implies that $r\left(K_{n,n}, K_{s,t}, st-\big\lfloor\frac{s+t}{2}\big\rfloor+1\right)=O\left(n^{(s+t-2)/\lfloor (s+t)/2\rfloor}\right)$, which is subquadratic in $n$. Combining this observation with the lower bound of Corollary~\ref{co:Kss+1} (see Section~\ref{sec:PfC_p}), we obtain
$cn^{2-2/\lfloor s/2\rfloor}\leq r\left(K_{n,n}, K_{s,t}, st-\big\lfloor\frac{s+t}{2}\big\rfloor+1\right)$ $\leq Cn^{2-1/s}$ for $t=s+1$, where $c$ and $C$ are two constants. We can also obtain the following lower bound for the case $t\geq s+2$.

\begin{theorem}\label{th:UnbalanceQu3+} For integers $s\geq 3$ and $t\geq s+2$, if $(s,t)\notin \{(3,5), (3,7)\}$, then $r\big(K_{n,n}, K_{s,t}, st-\big\lfloor\frac{s+t}{2}\big\rfloor+1\big)=\Omega(n^{4/3})$.
\end{theorem}

However, it seems difficult to determine whether $q=st-\big\lfloor\frac{s+t}{2}\big\rfloor+2$ is the threshold for quadratic $r(K_{n,n}, K_{s,t}, q)$. This is true for the case $s=t$; see Table~\ref{tab:rnn}. For the unbalanced case, we can only confirm this for $K_{s,s+1}$, $K_{2,t}$ and $K_{3,t}$ ($t$ is even); see Corollary~\ref{co:K2t3teven} below. Corollary~\ref{co:K2t3teven} is an immediate consequence of the following theorem.

\begin{theorem}\label{th:UnbalanceQu1} Let $t\geq s\geq 2$ be two integers. Then the following statements hold.
\begin{itemize}
\item[{\rm (i)}] $r(K_{n,n}, K_{s,t}, st-s+2)=\Theta(n^2)$.

\item[{\rm (ii)}] If $2\leq s\leq 3$, then $r\big(K_{n,n}, K_{s,t}, st-\big\lfloor\frac{t}{2}\big\rfloor+1\big)=\Theta(n^2)$.

\item[{\rm (iii)}] If $s\geq 4$ and at least one of $s$ and $t$ is even, then $r\big(K_{n,n}, K_{s,t}, st-\big\lfloor\frac{s+t}{2}\big\rfloor+3\big)= \Theta(n^{2})$.

\item[{\rm (iv)}] If $s\geq 5$ and both $s$ and $t$ are odd, then $r\big(K_{n,n}, K_{s,t}, st-\big\lfloor\frac{s+t}{2}\big\rfloor+4\big)= \Theta(n^{2})$.
\end{itemize}
\end{theorem}

\begin{corollary}\label{co:K2t3teven} The following statements hold.
\begin{itemize}
\item[{\rm (i)}] For any integer $s\geq 2$, we have $r\big(K_{n,n}, K_{s,s+1}, s(s+1)-\big\lfloor\frac{2s+1}{2}\big\rfloor+2\big)=\Theta(n^2)$.

\item[{\rm (ii)}] For any integer $t\geq 2$, we have $r\big(K_{n,n}, K_{2,t}, 2t-\big\lfloor\frac{2+t}{2}\big\rfloor+2\big)=\Theta(n^2)$.

\item[{\rm (iii)}] For any even integer $t\geq 4$, we have $r\big(K_{n,n}, K_{3,t}, 3t-\big\lfloor\frac{3+t}{2}\big\rfloor+2\big)=\Theta(n^2)$.
\end{itemize}
\end{corollary}

Theorem~\ref{th:UnbalanceQu1} implies that the threshold for quadratic $r(K_{n,n}, K_{s,t}, q)$ is between $st-\big\lfloor\frac{s+t}{2}\big\rfloor+2$ and $st-\big\lfloor\frac{s+t}{2}\big\rfloor+4$ when $t\geq s\geq 2$. For general integers $s, t$ and $q=st-\big\lfloor\frac{s+t}{2}\big\rfloor+2$, we can prove the following lower bound result.

\begin{theorem}\label{th:UnbalanceQu2} For integers $t\geq s\geq 3$, we have $r\big(K_{n,n}, K_{s,t}, st-\big\lfloor\frac{s+t}{2}\big\rfloor+2\big)=\Omega(n^{3/2})$.
\end{theorem}

We will prove Theorems~\ref{th:UnbalanceQu3+}, \ref{th:UnbalanceQu1} and \ref{th:UnbalanceQu2} using Corr\'{a}di's Lemma and its generalization (see Lemmas \ref{le:Corradi} and \ref{le:gCorradi}). We next provide some results we proved by combining Corr\'{a}di's Lemma and the Color Energy Method. The Color Energy Method is a recently developed technique for determining lower bounds on $f(n, p, q)$. This method was first introduced by Pohoata and Sheffer \cite{PoSh}, and further developed by Fish, Pohoata and Sheffer \cite{FiPS} and Balogh, English, Heath and Krueger \cite{BEHK}. In Section~\ref{subsec:colorenergy}, we will introduce a generalization of this method that is suited for studying $r(K_{n,n}, K_{s,t}, q)$. Here we continue with the overview of our main results.

Axenovich, F\"{u}redi and Mubayi \cite{AxFM} proved that $n^{4/3}-2n^{2/3}+1\leq r\left(K_{n,n}, K_{p,p}, p^2-p+1\right)\leq c'_pn^{2-2/p}$ for $p\geq 6$ and $n>p^{3/2}$. We improve the lower bound as follows.

\begin{theorem}\label{th:C_p} For any integer $p\geq 2$, we have $r\left(K_{n,n}, K_{p,p}, p^2-p+1\right)=\Omega\left(n^{2-\frac{2}{\lfloor p/2\rfloor}}\right)$.
\end{theorem}

In the case $p\geq 10$ and $p\equiv 2 \pmod{4}$, Theorem~\ref{th:C_p} is a special case of Theorem~\ref{th:r=2theta} below (with 2 substituted for $t$).

\begin{theorem}\label{th:r=2theta} For any integer $t\geq 2$ and odd number $s\geq 5$, let $p=(s-1)t+2$. Then
$r\left(K_{n,n}, K_{p,p}, p^2-st+1)\right)=\Omega\big(n^{2-2/s}\big).$
\end{theorem}

Note that the upper bound on $r\left(K_{n,n}, K_{(s-1)t+2,(s-1)t+2}, ((s-1)t+2)^2-st+1\right)$ given by (\ref{eq:UnbalanGener}) is $O\Big(n^{2-\frac{2}{s}+\frac{2}{st}}\Big)$, which is arbitrarily close to our lower bound for any fixed $s$ and sufficiently large $t$.

For any integers $t\geq 2$, $r\geq 3$ and odd number $s\geq 2r-1$, let $p=r((s-1)t/2+1)$. Using the upper bound (\ref{eq:UnbalanGener}), we obtain
$r\left(K_{n,n}, K_{p,p}, p^2-(r-1)st+2\right)=O\Big(n^{\frac{r(s-1)t+2r-2}{(r-1)st-1}}\Big)$. By straightforward calculations, it can be shown that $\frac{r(s-1)t+2r-2}{(r-1)st-1}<\frac{r}{r-1}\left(1-\frac{1}{s}+\frac{2}{st}\right)$. Hence, the following result provides a lower bound that is arbitrarily close to the upper bound for sufficiently large $t$.

\begin{theorem}\label{th:r3theta} For integers $t\geq 2$, $r\geq 3$ and odd number $s\geq 2r-1$, let $p=r((s-1)t/2+1)$. Then $r\left(K_{n,n}, K_{p,p}, p^2-(r-1)st+2\right)=\Omega\left(n^{\frac{r}{r-1}\left(1-\frac{1}{s}\right)}\right)$.
\end{theorem}

Using (\ref{eq:UnbalanGener}), we can also obtain that $r\left(K_{n,n}, K_{2t,t(t-1)}, 2t^2(t-1)-t(t-1)+1\right)=O\big(n^{1+2/t}\big)$ for any integer $t\geq 2$. We will prove the following lower bound result.

\begin{theorem}\label{th:Kt+} For any integer $t\geq 2$, we have $r\left(K_{n,n}, K_{2t,t(t-1)}, 2t^2(t-1)-t(t-1)+1\right)= \Omega\left(n^{1+\frac{1}{2t-3}}\right)$.
\end{theorem}

Similarly, using (\ref{eq:UnbalanGener}), we have that $r\left(K_{n,n}, K_{s,st}, s^2t-t(s-1)+1\right)= O\left(n^{1+\frac{1}{t}+\frac{t-1}{st-t}}\right)$ for any integers $s, t\geq 2$. We will obtain the following lower bound for these choices of the parameters.

\begin{theorem}\label{th:Ksst} For integers $s, t\geq 2$, we have $r\left(K_{n,n}, K_{s,st}, s^2t-t(s-1)+1\right)=\Omega\big(n^{1+1/t}\big)$.
\end{theorem}

Finally, we provide some results which will be proved without using the Color Energy Method. Axenovich, F\"{u}redi and Mubayi \cite{AxFM} proved that $r\left(K_{n,n}, K_{p,p}, p^2-2p+2\right)=O\big(n^{1-1/(2p-1)}\big)$ for $p\geq 2$. We will prove the following lower bound result.

\begin{theorem}\label{th:2p-2} For any integer $p\geq 2$, we have $r\left(K_{n,n}, K_{p,p}, p^2-2p+2\right)=\Omega\big(n^{1-1/p}\big)$.
\end{theorem}

Let $s, t, a$ and $b$ be four integers with $2\leq a\leq s$, $2\leq b\leq t$ and $ab\geq s+t$. By Theorem~\ref{th:UnbalanceLin} (iii), we have $r(K_{n,n}, K_{s,t}, st-ab+2)=o(n)$. We will obtain the following lower bound using the result of K\H{o}v\'{a}ri, S\'{o}s and Tur\'{a}n \cite{KoST} that ex$(n, n, K_{a,b})= O\left(n^{2-1/a}\right)$ for $b\geq a\geq 1$.

\begin{theorem}\label{th:Kab*} For integers $s, t, a$ and $b$ with $2\leq a\leq s$, $2\leq b\leq t$ and $ab\geq s+t$, we have $r(K_{n,n}, K_{s,t}, st-ab+2)=\Omega\big(n^{1/\min\{a,b\}}\big)$.
\end{theorem}

In the case $a\leq b$ and $a+b\leq s+t-2$, the following result improves Theorem~\ref{th:Kab*}. To see this, note that $st-a(s+t-a-2)+1<st-ab+2$ when $a+b\leq s+t-2$. Recall that we have the property that $r(K_{n,n}, K_{s,t}, k)\leq r(K_{n,n}, K_{s,t}, \ell)$ for $2\leq k\leq \ell\leq st$. Thus Theorem~\ref{th:induction} gives the same lower bound for a smaller number of colors on each $K_{s,t}$.

\begin{theorem}\label{th:induction} For integers $s, t$ and $a$ with $2\leq a\leq s\leq t$ and $a(s+t-a-2)\geq s+t-1$, we have $r\big(K_{n,n}, K_{s,t}, st-a(s+t-a-2)+1\big)=\Omega\big(n^{1/a}\big)$.
\end{theorem}

As corollaries of Theorems~\ref{th:Kab*} and \ref{th:induction}, we obtain the following results for balanced complete bipartite graphs $K_{p,p}$.

\begin{corollary}\label{co:Kab} Let $p, a$ and $b$ be three integers with $2\leq a\leq b\leq p$.
\begin{itemize}
\item[{\rm (i)}] If $ab\geq 2p$, then $r\left(K_{n,n}, K_{p,p}, p^2-ab+2\right)=\Omega\big(n^{1/a}\big)$.
\item[{\rm (ii)}] If $a(2p-a-2)\geq 2p-1$, then $r\left(K_{n,n}, K_{p,p}, p^2-a(2p-a-2)+1\right)=\Omega\big(n^{1/a}\big)$.
\end{itemize}
\end{corollary}

The remainder of this paper is organized as follows. In the next section, we provide some additional terminology and results that will be used in our proofs. There we also introduce the Color Energy Method and present our extension to bipartite graphs, together with our proof of Theorem~\ref{th:Kt+}. In Section~\ref{sec:PfC_p}, we provide our proof of Theorem~\ref{th:C_p}. In Section~\ref{sec:adv}, we further develop the technique used in Sections~\ref{sec:pre} and \ref{sec:PfC_p}, and prove Theorems~\ref{th:r=2theta} and \ref{th:r3theta} using this advanced technique. Section~\ref{sec:unbalanced} is devoted to our proofs of Theorems~\ref{th:UnbalanceLin}, \ref{th:UnbalanceQu3+}, \ref{th:UnbalanceQu1} and \ref{th:UnbalanceQu2}. In Section~\ref{sec:unbalanced2}, we give two proofs of Theorem~\ref{th:Ksst}, one using the Color Energy Method and the other one without using the Color Energy Method, and we compare the two proofs. In Section~\ref{sec:not-energy}, we prove Theorems~\ref{th:2p-2}, \ref{th:Kab*} and \ref{th:induction}. In Section~\ref{sec:concluding}, we provide some open problems, and give two explanations to illustrate the differences between studying $f(n, p, q)$ and $r(K_{n,n}, K_{s,t}, q)$ using the Color Energy Method. In Appendix~\ref{ap:1}, we prove Lemma~\ref{le:UnbalanceQu3+}, which will be used in our proof of Theorem~\ref{th:UnbalanceQu3+}. In Appendices~\ref{ap:O1} and \ref{ap:On}, we provide some results on the thresholds for $r(K_{n,n}, K_{s,t}, q)=n^2-c$ and $r(K_{n,n}, K_{s,t}, q)=n^2-O(n)$, respectively.


\section{Preliminaries}\label{sec:pre}

We begin with some additional terminology and notation. For a positive integer $n$, let $[n]\colonequals \{1, 2, \ldots, n\}$. Given a graph $G$ and an integer $k \geq 1$, let $c\colon\, E(G)\rightarrow [k]$ be a $k$-edge-coloring (not necessarily a proper edge-coloring) of $G$. For an edge $e\in E(G)$, let  $c(e)$ be the color assigned to edge $e$. We denote by $C(G)$ the set of all colors used on the edges of $G$, i.e., $C(G)=\{c(e)\colon\, e\in E(G)\}$. For two disjoint nonempty subsets $U$, $V\subset V(G)$, let $C(U, V)=\{c(uv)\colon\, uv\in E(G), u\in U, v\in V\}$. If $U$ consists of a single vertex $u$, then we simply write $C(\{u\}, V)$ as $C(u, V)$. For any vertex $v\in V(G)$ and color $i\in [k]$, we say that \emph{$v$ is incident with color $i$} if there exists an edge $e$ with $c(e)=i$ such that $v$ is incident with $e$. For a subset $U\subseteq V(G)$, let $G[U]$ denote the subgraph of $G$ induced by $U$. For a color $i\in [k]$, the subgraph \emph{induced by color $i$} is the subgraph consisting of all the edges with color $i$ and all the vertices that are incident with color $i$. Given a bipartite graph $G$ with partite sets $A$ and $B$, we also use $G(A,B)$ to denote this bipartite graph.

In order to study lower bounds on $r(G, H, q)$, it is convenient to consider the concept of \emph{color repetition}. Given an edge-colored graph $F$, the number of color repetitions in $F$ is defined to be $|E(F)|-|C(F)|$. In other words, the statement that $F$ has at least $x$ color repetitions is equivalent to the statement that $F$ is colored by at most $|E(F)|-x$ distinct colors.

For any graph $H$ and positive integer $n$, the \emph{Tur\'{a}n number} of $H$, denoted by ex$(n, H)$, is the maximum number of edges in an $n$-vertex $H$-free graph. For any bipartite graph $H$ and positive integers $m, n$, the \emph{bipartite Tur\'{a}n number} of $H$, denoted by ex$(m, n, H)$, is the maximum number of edges in an $H$-free graph $G$, where $G$ is a spanning subgraph of $K_{m,n}$. Note that ex$(m, n, H)\leq $ ex$(m+n, H)$ for any bipartite graph $H$. For any positive integers $m, n, a$ and $b$, the \emph{Zarankiewicz function} z$(m, n; a, b)$ is defined to be the maximum number of edges in a spanning subgraph $G$ of $K_{m,n}$ such that $G$ contains no $K_{a,b}$ with $a$ vertices in the partite set of size $m$ and $b$ vertices in the partite set of size $n$. For any integer $t\geq 3$, let $K^{1}_{t}$ be the \emph{subdivision} of $K_t$, i.e., the graph obtained from $K_t$ by replacing each edge with a path of length 2. For integers $a, b\geq 2$, the \emph{theta graph} $\mit\Theta$$(a,b)$ consists of two vertices connected by $b$ internally disjoint paths of length $a$. We shall use the following results.

\begin{theorem}\label{th:Tu} The following results have been established.
\begin{itemize}
\item[{\rm (i)}]{\normalfont (\cite{Jan})} For any integer $t\geq 3$, we have {\rm ex}$(n, K^{1}_{t})=O\left(n^{3/2-1/(4t-6)}\right)$.
\item[{\rm (ii)}]{\normalfont (\cite{BuTa,FaSi})} For integers $a, b\geq 2$, we have {\rm ex}$(n, \mit\Theta$$(a, b))=O\left(n^{1+1/a}\right)$.
\item[{\rm (iii)}]{\normalfont (\cite{NaVe})} For any integer $k\geq 2$, we have {\rm ex}$(n, n, C_{2k})=O\left(n^{1+1/k}\right)$.
\item[{\rm (iv)}]{\normalfont (\cite{KoST}, see also \cite[Section~IV.2]{Bol})} For positive integers $m, n, a$ and $b$, we have $\mbox{z}(m, n; a, b)\leq (b-1)^{1/a}(m-a+1)n^{1-1/a}+(a-1)n$.
\end{itemize}
\end{theorem}

We will also use the following two known combinatorial lemmas which can also be found in \cite[Section~2.1]{Juk}.

\begin{lemma}\label{le:Erd64}{\normalfont (\cite{Erd64})} Let $A$ be a set of $n$ elements and let $t\geq 2$ be an integer. Let $A_1, A_2, \ldots, A_k$ be subsets of $A$ of average size at least $m$. If $k\geq 2t(n/m)^t$, then there exist $1\leq j_1<j_2<\cdots <j_t\leq k$ such that $|A_{j_1}\cap A_{j_2}\cap \cdots \cap A_{j_t}|\geq m^t/(2n^{t-1})$.
\end{lemma}

\begin{lemma}\label{le:Corradi}{\normalfont (Corr\'{a}di's Lemma \cite{Cor})} Let $X_1, X_2, \ldots, X_m$ be $m$ sets with $|X_i|\geq a$ for all $i\in [m]$. If $|X_i\cap X_j|\leq \ell$ for all $i\neq j$, then $|X_1\cup X_2\cup \cdots \cup X_m|\geq a^2m/(a+ (m- 1)\ell)$.
\end{lemma}

In some of our proofs, we will apply the following generalization of Corr\'{a}di's Lemma.

\begin{lemma}\label{le:gCorradi} Let $2\leq r\leq m$, and $X_1, X_2, \ldots, X_m$ be $m$ sets with $|X_i|\geq a$ for all $i\in [m]$. If $|X_{j_1}\cap \cdots \cap X_{j_r}|\leq \ell$ for all $1\leq j_1< \cdots < j_r\leq m$, then $|X_1\cup X_2\cup \cdots \cup X_m|\geq \Big(a^rm^{r-1}/\Big(a\Big(m^{r-1}-\frac{(m-1)!}{(m-r)!}\Big)+\frac{(m-1)!}{(m-r)!}\ell\Big)\Big)^{1/(r-1)}$.
\end{lemma}

\begin{proof} For each $i\in [m]$, let $Y_i\subseteq X_i$ with $|Y_i|=a$. Note that $|Y_{j_1}\cap \cdots \cap Y_{j_r}|\leq |X_{j_1}\cap \cdots \cap X_{j_r}|\leq \ell$ for all $1\leq j_1< \cdots < j_r\leq m$. Let $Y=Y_1\cup Y_2\cup \cdots \cup Y_m$. For any $x\in Y$, let $d(x)=|\{Y_j\colon\, x\in Y_j, j\in [m]\}|$. We first prove the following claim.

\begin{claim}\label{cl:gCorradi-1} For any $Y'\subseteq Y$ and $1\leq t \leq m$, we have $\sum_{x\in Y'}(d(x))^{t}=\sum_{(j_1, \ldots, j_{t})\in [m]^{t}}|Y_{j_1}\cap \cdots \cap Y_{j_{t}}\cap Y'|$.
\end{claim}

\begin{proof} We will apply double counting on the number of edges of an auxiliary graph $H$ defined as follows. Let $H$ be the bipartite graph with bipartition $Y'$ and $\mathcal{Y}=\{Y_{j_1}\cap \cdots \cap Y_{j_{t}} \colon\, (j_1, \ldots, j_{t})\in [m]^{t}\}$ such that $x\in Y'$ and $Y_{j_1}\cap \cdots \cap Y_{j_{t}}\in \mathcal{Y}$ (note that $Y_{j_1}\cap \cdots \cap Y_{j_{t}}$ is a vertex of $H$) are adjacent in $H$ if and only if $x\in Y_{j_1}\cap \cdots \cap Y_{j_{t}}$. Note that for any $x\in Y'$, the degree of $x$ in $H$ is $(d(x))^{t}$. Thus $|E(H)|=\sum_{x\in Y'}(d(x))^{t}$. On the other hand, for any $Y_{j_1}\cap \cdots \cap Y_{j_{t}}\in \mathcal{Y}$, its degree in $H$ is $|Y_{j_1}\cap \cdots \cap Y_{j_{t}}\cap Y'|$. Thus $|E(H)|=\sum_{(j_1, \ldots, j_{t})\in [m]^{t}}|Y_{j_1}\cap \cdots \cap Y_{j_{t}}\cap Y'|$. So $\sum_{x\in Y'}(d(x))^{t}=\sum_{(j_1, \ldots, j_{t})\in [m]^{t}}|Y_{j_1}\cap \cdots \cap Y_{j_{t}}\cap Y'|$.
\end{proof}

By Claim \ref{cl:gCorradi-1} and since $|Y_{j_1}\cap \cdots \cap Y_{j_r}|\leq \ell$ for all $1\leq j_1< \cdots < j_r\leq m$, we have
\begin{align}\label{al:gCorradi-1}
   \sum_{x\in Y}(d(x))^r=&~\sum_{(j_1, \ldots, j_{r})\in [m]^{r}}|Y_{j_1}\cap \cdots \cap Y_{j_{r}}|\nonumber\\
  = &~\sum_{\substack{(j_1, \ldots, j_{r})\in [m]^{r},  \\ j_1, \ldots, j_{r}~\text{are not pairwise distinct}}}|Y_{j_1}\cap \cdots \cap Y_{j_{r}}| + \sum_{\substack{(j_1, \ldots, j_{r})\in [m]^{r},  \\ j_1, \ldots, j_{r}~\text{are pairwise distinct}}}|Y_{j_1}\cap \cdots \cap Y_{j_{r}}|\nonumber\\
  \leq &~ a\left(m^r-\frac{m!}{(m-r)!}\right)+\frac{m!}{(m-r)!}\ell.
\end{align}
Using Jensen's inequality, we have
\begin{align}\label{al:gCorradi-2}
  \sum_{x\in Y}(d(x))^r\geq &~\frac{1}{|Y|^{r-1}}\Bigg(\sum_{x\in Y}d(x)\Bigg)^r= \frac{1}{|Y|^{r-1}}\Bigg(\sum_{j\in [m]}|Y_j|\Bigg)^r= \frac{(am)^r}{|Y|^{r-1}}.
\end{align}
Combining inequalities (\ref{al:gCorradi-1}) and (\ref{al:gCorradi-2}), we have $|X_1\cup X_2\cup \cdots \cup X_m|\geq |Y|\geq \Big((am)^r/\Big(a\Big(m^r-\frac{m!}{(m-r)!}\Big)+\frac{m!}{(m-r)!}\ell\Big)\Big)^{1/(r-1)}=\Big(a^rm^{r-1}/\Big(a\Big(m^{r-1}- \frac{(m-1)!}{(m-r)!}\Big)+\frac{(m-1)!}{(m-r)!}\ell\Big)\Big)^{1/(r-1)}$.
\end{proof}

To prove Theorem~\ref{th:UnbalanceQu3+}, we will use the following lemma, which is proved in Appendix \ref{ap:1}.

\begin{lemma}\label{le:UnbalanceQu3+} Let $s, t$ be integers with $s\geq 3$, $t\geq 3s-2$ and $(s,t)\neq (3,7)$. If $\big\lfloor\frac{s+t}{2}\big\rfloor-s+1$ is even, then $\frac{3}{2}\big(\big\lfloor\frac{s+t}{2}\big\rfloor-s+1\big)+s\leq t$. If $\big\lfloor\frac{s+t}{2}\big\rfloor-s+1$ is odd, then $\frac{3}{2}\big(\big\lfloor\frac{s+t}{2}\big\rfloor-s+2\big)+s-1\leq t$.
\end{lemma}

The rest of this section is devoted to a description of our extension of the Color Energy Method to bipartite graphs (in Section~\ref{subsec:colorenergy}), followed by an illustration how to apply this extension of the Color Energy Method by proving Theorem~\ref{th:Kt+} (in Section~\ref{subsec:example}).


\subsection{Color Energy Method}\label{subsec:colorenergy}

Motivated by the additive energy in additive combinatorics, Pohoata and Sheffer \cite{PoSh} defined the color energy of an edge-colored graph. Using this new tool, they studied the Erd\H{o}s-Gy\'{a}rf\'{a}s function and a problem related to the Erd\H{o}s distinct distances problem (see \cite{Erd46,GuKa}) in discrete geometry. In \cite{FiPS}, Fish, Pohoata and Sheffer introduced the concept of higher color energies. Recently, Balogh, English, Heath and Krueger \cite{BEHK} further developed the Color Energy Method and applied it to establish various lower bounds on the Erd\H{o}s-Gy\'{a}rf\'{a}s function. In this subsection, we will extend the Color Energy Method to study the bipartite Erd\H{o}s-Gy\'{a}rf\'{a}s function. We remark that the generalization of this method from complete graphs to complete bipartite graphs (as the host graph) is not trivial. We will give two examples to illustrate this in the concluding Section~\ref{sec:concluding}.

We start by giving our definition of the color energy and the color energy graph of an edge-colored bipartite graph.

\begin{definition}\label{def:energy_graph} {\normalfont Let $G=G(A, B)$ be a copy of $K_{n,n}$ with an edge-coloring $c\colon\, E(K_{n,n})\to C(G)$. For an integer $r\geq 2$, the \emph{$r$th color energy} of $G$ is defined to be $\mathbb{E}_r(G)\colonequals |\{(a_1, \ldots, a_r,$ $b_1, \ldots, b_r)\in A^r\times B^r\colon\, c(a_1b_1)=\cdots =c(a_rb_r)\}|$. The \emph{$r$th color energy graph $G^r$} of $G$ is a bipartite graph with partite sets $A^r$ and $B^r$, in which there is an edge between $(a_1, \ldots, a_r)$ and $(b_1, \ldots, b_r)$ if and only if $c(a_1b_1)=\cdots =c(a_rb_r)$.}
\end{definition}

Note that $\mathbb{E}_r(G)=|E(G^r)|$. Using this, we can give the following lower bound on the number of colors used on the edges of $G$, which is an expression in terms of the number of edges of $G^r$.

\begin{proposition}\label{prop:|C(G)|} $|C(G)|\geq \left(\frac{n^{2r}}{|E(G^r)|}\right)^{\frac{1}{r-1}}$.
\end{proposition}

\begin{proof} For each color $i\in C(G)$, let $m_i$ be the number of edges with color $i$ in $G$. Then $\sum_{i\in C(G)}m_i=n^2$ and $|E(G^r)|=\sum_{i\in C(G)}m_i^r \geq |C(G)|\left(\frac{\sum_{i\in C(G)}m_i}{|C(G)|}\right)^r =\frac{n^{2r}}{|C(G)|^{r-1}}$. The result follows.
\end{proof}

Note that for any edge $\vec{a}\vec{b}\in E(G^r)$ with $\vec{a}=(a_1, \ldots, a_r)$ and $\vec{b}=(b_1, \ldots, b_r)$, the edges $a_1b_1, \ldots, a_rb_r$ are colored by the same color in $G$, and we denote this color by $c(\vec{a}\vec{b})$. For any subgraph $\mbox{sub}(G^{r})\subseteq G^r$, let $\mathcal{C}(\mbox{sub}(G^{r}))\colonequals \{i\in C(G)\colon\, c(\vec{a}\vec{b})=i \mbox{ for some } \vec{a}\vec{b}\in E(\mbox{sub}(G^{r}))\}$. We will prune the color energy graph $G^r$ in several steps. Before we can define what we mean by the pruned $r$th color energy graph, we first need the following three propositions.

\begin{proposition}\label{prop:prune} There exists a partition $A_1, \ldots, A_r$ of $A$ {\rm(}resp., $B_1, \ldots, B_r$ of $B${\rm)} with $|A_i|, |B_i|\in \{\lfloor n/r\rfloor, \lceil n/r\rceil\}$ for all $i\in [r]$, such that the subgraph $\widehat{G}^r$ of $G^r$ induced by $(A_1\times \cdots \times A_r)\cup (B_1\times \cdots \times B_r)$ satisfies $|E(\widehat{G}^r)|=\Theta(|E(G^r)|)$.
\end{proposition}

\begin{proof} Choose a partition $A_1, \ldots, A_r$ of $A$ (resp., $B_1, \ldots, B_r$ of $B$) among the set of partitions into $r$ parts of size $\lfloor n/r\rfloor$ or $\lceil n/r\rceil$ uniformly at random. We assume here that we choose the partitions of $A$ and $B$ independently. Let $\widehat{G}^r$ be the subgraph of $G^r$ induced by $(A_1\times \cdots \times A_r)\cup (B_1\times \cdots \times B_r)$. Let $X=|E(\widehat{G}^r)|$. For any edge $(a_1, \ldots, a_r)(b_1, \ldots, b_r)$ of $G^r$, the probability that $a_i\in A_i$ and $b_i\in B_i$ for all $i\in [r]$ is at least $(1-o(1))(1/r)^{2r}$. Thus the expectation of $X$ is at least $(1-o(1))|E(G^r)|(1/r)^{2r}$. Hence, there exists a partition as claimed in the proposition.
\end{proof}

\begin{proposition}\label{prop:prunelogn} Assume that we have $|C(G)|=o(n^2/\log n)$. Then $\widehat{G}^r$ contains a subgraph $\bar{G}^r$ such that {\rm (i)} for any color $i\in \mathcal{C}(\bar{G}^r)$, there are at least $\log n$ distinct edges in $G$ with color $i$; {\rm (ii)} $|E(\bar{G}^r)|=\Theta(|E(\widehat{G}^r)|)$.
\end{proposition}

\begin{proof} For every color $i\in C(G)$ that appears fewer than $\log n$ times in $G$, we delete the edges associated with color $i$ from $\widehat{G}^r$. Let $\bar{G}^r$ denote the resulting graph. Note that $|E(\widehat{G}^r)\setminus E(\bar{G}^r)|< |C(G)|\log^r n$. By Propositions~\ref{prop:|C(G)|} and \ref{prop:prune}, we have $|E(\widehat{G}^r)|= \Omega(n^{2r}/|C(G)|^{r-1})$. Since $(|C(G)|\log^r n)/(n^{2r}/|C(G)|^{r-1})=((|C(G)|\log n)/n^2)^r=o(1)$, we have $|E(\bar{G}^r)|= \Theta(|E(\widehat{G}^r)|)$.
\end{proof}

\begin{proposition}\label{prop:pruneP3} Assume that $G$ contains no monochromatic star $K_{1,\ell}$ for some constant $\ell>0$. Then $\bar{G}^r$ contains a subgraph $\widetilde{G}^r$ such that {\rm (i)} if $\vec{x}, \vec{y}\in V(\widetilde{G}^r)$ have a common neighbor in $\widetilde{G}^r$, then $\vec{x}$ and $\vec{y}$ are not equal in any coordinate; {\rm (ii)} $|E(\widetilde{G}^r)|=\Theta(|E(\bar{G}^r)|)$.
\end{proposition}

\begin{proof} Note that $\bar{G}^r$ is a bipartite graph with partite sets $A_1\times \cdots \times A_r$ and $B_1\times \cdots \times B_r$. For any $\vec{a}=(a_1, \ldots, a_r)\in A_1\times \cdots \times A_r$, let $d(\vec{a})$ be the degree of $\vec{a}$ in $\bar{G}^r$. For any $u\in B$ such that $\vec{a}$ has a neighbor with first coordinate $u$ in $\bar{G}^r$, let $N_u(\vec{a})=\{\vec{b}=(u, b_2, \ldots, b_r): \vec{a}\vec{b}\in E(\bar{G}^r)\}$. Let $c_0\colonequals c(a_1u)$. Then $c(a_2b_2)=\cdots =c(a_rb_r)=c_0$ for any $b_2, \ldots, b_r$ with $(u, b_2, \ldots, b_r)\in N_u(\vec{a})$. Since $G$ contains no monochromatic star $K_{1,\ell}$, we have $|N_u(\vec{a})|\leq (\ell-1)^{r-1}$, i.e., $\vec{a}$ has at most $(\ell-1)^{r-1}$ neighbors with first coordinate $u$. Similarly, for any $k\in [r]$ and $v\in B$, $\vec{a}$ has at most $(\ell-1)^{r-1}$ neighbors with $k$th coordinate $v$. We perform the following $r$-step operation to $\vec{a}$: in the first step, for each vertex $v\in B$ such that $\vec{a}$ has a neighbor with first coordinate $v$, we keep one edge of all edges joining $\vec{a}$ and its neighbors with first coordinate $v$; in the $k$th step ($2\leq k\leq r$), for each vertex $v\in B$ such that $\vec{a}$ has a neighbor with $k$th coordinate $v$ in the resulting subgraph of the $(k-1)$th step, we keep one edge of all edges joining $\vec{a}$ and its neighbors with $k$th coordinate $v$. We perform this operation for all $\vec{a}\in A_1\times \cdots \times A_r$ one-by-one, and let $\widetilde{G}^r_0$ be the remaining subgraph of $\bar{G}^r$. Let $d'(\vec{a})$ be the degree of $\vec{a}$ in $\bar{G}^r_0$. Then $d'(\vec{a})\geq d(\vec{a})/(\ell-1)^{r(r-1)}$. Thus $|E(\widetilde{G}^r_0)|\geq |E(\bar{G}^r)|/(\ell-1)^{r(r-1)}= \Theta(|E(\bar{G}^r)|)$. For each $\vec{b}\in B_1\times \cdots \times B_r$, we perform an analogous operation in $\widetilde{G}^r_0$. Let $\widetilde{G}^r$ be the resulting subgraph of $\widetilde{G}^r_0$. Then $|E(\widetilde{G}^r)|\geq |E(\widetilde{G}^r_0)|/(\ell-1)^{r(r-1)}=\Theta(|E(\bar{G}^r)|)$. Moreover, from the construction, $\widetilde{G}^r$ satisfies condition (i).
\end{proof}

We now give our definition of the pruned $r$th color energy graph. Note that Propositions~\ref{prop:prune}, \ref{prop:prunelogn} and \ref{prop:pruneP3} guarantee the existence of such a graph.

\begin{definition}\label{def:prune energy_graph} {\normalfont Let $G=G(A, B)$ be a copy of $K_{n,n}$ with an edge-coloring $c: E(K_{n,n})\to C(G)$ such that $|C(G)|=o(n^2/\log n)$ and $G$ contains no monochromatic star $K_{1,\ell}$ for some constant $\ell>0$. For an integer $r\geq 2$, the \emph{pruned $r$th color energy graph $\widetilde{G}^r$} of $G$ is a subgraph of $G^r$ with the following properties:
\begin{itemize}
\item[{\rm (i)}] There exists a partition $A_1, \ldots, A_r$ of $A$ {\rm (}resp., $B_1, \ldots, B_r$ of $B${\rm )} with $|A_i|, |B_i| \in \{\lfloor n/r\rfloor, \lceil n/r\rceil\}$ for all $i\in [r]$, such that $V(\widetilde{G}^r)= (A_1\times \cdots \times A_r)\cup (B_1\times \cdots \times B_r)$.
\item[{\rm (ii)}] For any color $i\in \mathcal{C}(\widetilde{G}^r)$, there are at least $\log n$ edges in $G$ with color $i$.
\item[{\rm (iii)}] If $\vec{x}, \vec{y}$ have a common neighbor in $\widetilde{G}^r$, then $\vec{x}$ and $\vec{y}$ are not equal in any coordinate.
\item[{\rm (iv)}] $|E(\widetilde{G}^r)|=\Theta(|E(G^r)|)$.
\end{itemize}
}
\end{definition}

By Proposition \ref{prop:|C(G)|} and the above property (iv), we have
\begin{equation}\label{eq:color_edge}
|C(G)|=\Omega\Bigg(\left(\frac{n^{2r}}{|E(\widetilde{G}^r)|}\right)^{\frac{1}{r-1}}\Bigg).
\end{equation}
In order to establish a lower bound on $|C(G)|$, it suffices to prove an upper bound on $|E(\widetilde{G}^r)|$. A possible strategy for this is the following. Suppose $H$ is a bipartite graph with bipartite Tur\'{a}n number ex$(m,m, H)=O(m^{2-\alpha})$. If we can prove that $\widetilde{G}^r$ is $H$-free, then this implies $|E(\widetilde{G}^r)|=O((|V(\widetilde{G}^r)|/2)^{2-\alpha})=O(n^{(2-\alpha)r})$. In turn, using (\ref{eq:color_edge}), this would give us that $|C(G)|=\Omega\big(n^{\frac{\alpha r}{r-1}}\big)$. In the proofs of Theorems \ref{th:C_p}, \ref{th:r=2theta}, \ref{th:r3theta} and \ref{th:Kt+}, we shall prove that $\widetilde{G}^r$ (with appropriate $r$) contains no even cycle $C_{2\lfloor p/2 \rfloor}$, theta graph $\mit\Theta$$(s, 4s^2t)$, theta graph $\mit\Theta$$(s, 2rs^2t)$ and subdivision $K^{1}_{t}$ of $K_t$, respectively.

Given a subgraph $H$ of $\widetilde{G}^r$, the \emph{corresponding structure of $H$} in $G$ is a subgraph of $G$ defined as follows. The vertex set is $\{v\in V(G)\colon\, v \mbox{\ is\ a\ coordinate\ of\ some\ vertex\ of\ }H\}$. Two vertices $u\in A$ and $v\in B$ are adjacent if there exist an edge $\vec{x}\vec{y}\in E(H)$ and an index $j\in [r]$ such that $u$ is the $j$th coordinate of $\vec{x}$ and $v$ is the $j$th coordinate of $\vec{y}$. Note that this subgraph of $G$ is a bipartite graph with at most $r|E(H)|$ edges, whose two partite sets have sizes at most $r|V(H)\cap A^r|$ and $r|V(H)\cap B^r|$, respectively.

In the next subsection, we illustrate how to apply the above extension of the Color Energy Method by proving Theorem~\ref{th:Kt+}.


\subsection{Proof of Theorem~\ref{th:Kt+}}\label{subsec:example}

The case $t=2$ follows from Theorem~\ref{th:UnbalanceQu1} (ii) which will be proved in Section~\ref{sec:unbalanced}. The case $t=3$ follows from a result of Axenovich, F\"{u}redi and Mubayi (see \cite[Theorem~4.3]{AxFM}). Hence, we may assume that $t\geq 4$ in the following argument. Let $G=G(A, B)$ be a copy of $K_{n,n}$ with an edge-coloring $c\colon\, E(K_{n,n}) \to C(G)$, in which every $K_{2t,t(t-1)}$ receives at least $2t^2(t-1)-t(t-1)+1$ distinct colors. Equivalently, every $K_{2t,t(t-1)}$ receives at most $t(t-1)-1$ color repetitions. Let $A=\{a_1, a_2, \ldots, a_n\}$ and $B=\{b_1, b_2, \ldots, b_n\}$. Our goal is to prove $|C(G)|=\Omega(n^{1+1/(2t-3)})$. We may assume $|C(G)|=o(n^{1+1/(2t-3)})$; otherwise we are done.

We first show that $G$ contains no monochromatic $K_{1, t(t-1)-2}$. For a contradiction, suppose that $\{b_1, a_1, a_2, \ldots, a_{t(t-1)-2}\}$ forms a monochromatic star. For each $j\in [n]$, let $X_j\colonequals \{i\in C(G)\colon\, c(a_jb_{\ell})=i \mbox{ for some } 2\leq \ell \leq n\}$. Suppose that $|X_j|< (n-1)/3$ for some $j\in [n]$. Then there exist four edges of the same color incident with $a_j$. These four edges together with the monochromatic $K_{1, t(t-1)-2}$ form a subgraph of $K_{2t,t(t-1)}$ with at least $t(t-1)-3+3=t(t-1)$ color repetitions, a contradiction. Hence, we have $|X_j|\geq (n-1)/3$ for every $j\in [n]$. Suppose that $|X_{j_1}\cap X_{j_2}|\geq 3$ for some $1\leq j_1<j_2\leq n$. Since $2t\geq 8$, there exists a subgraph of $K_{2t,t(t-1)}$ with at least $t(t-1)-3+3=t(t-1)$ color repetitions, a contradiction. Hence, we have $|X_{j_1}\cap X_{j_2}|\leq 2$ for every $1\leq j_1<j_2\leq n$. By Lemma~\ref{le:Corradi}, we have $|C(G)|\geq |X_1\cup X_2\cup \cdots \cup X_n|\geq \frac{((n-1)/3)^2n}{(n-1)/3+(n-1)2}=\Omega(n^2)$, contradicting our assumption that $|C(G)|=o(n^{1+1/(2t-3)})$. Therefore, $G$ contains no monochromatic $K_{1, t(t-1)-2}$.

Hence, the pruned second color energy graph $\widetilde{G}^2$ exists. By Theorem~\ref{th:Tu} (i) and the arguments at the end of Section~\ref{subsec:colorenergy}, it suffices to prove that $\widetilde{G}^2$ is $K^1_t$-free. Suppose for a contradiction that $\widetilde{G}^2$ contains a copy $\widetilde{H}$ of $K^1_t$. Note that $\widetilde{H}$ is a $(t(t-1))$-edge bipartite graph whose two partite sets have sizes $t$ and $t(t-1)/2$, respectively. Without loss of generality, Let $V(\widetilde{H})=\{\vec{a}_1, \ldots, \vec{a}_t\}\cup \{\vec{b}_{1,2}, \ldots, \vec{b}_{1,t}, \vec{b}_{2,3}, \ldots, \vec{b}_{2,t}, \ldots, \vec{b}_{t-1,t}\}$, where $\vec{a}_i=(a^{(1)}_i, a^{(2)}_i)$, $a^{(1)}_i\in A_1$, $a^{(2)}_i\in A_2$ for all $i\in [t]$, $\vec{b}_{i,j}=(b^{(1)}_{i,j}, b^{(2)}_{i,j})$, $b^{(1)}_{i,j}\in B_1$, $b^{(2)}_{i,j}\in B_2$ for all $1\leq i<j\leq t$, and $\vec{b}_{i,j}$ is adjacent to $\vec{a}_i$ and $\vec{a}_j$ in $\widetilde{H}$ for every $1\leq i<j\leq t$. For any $1\leq i< j\leq t$, since $\vec{a}_i$ and $\vec{a}_j$ have a common neighbor $\vec{b}_{i,j}$ in $\widetilde{H}$, the two vertices $\vec{a}_i$ and $\vec{a}_j$ are not equal in any coordinate by Definition~\ref{def:prune energy_graph} (iii).

Note that the corresponding structure of $\widetilde{H}$ in $G$ is a subgraph of $K_{2t,t(t-1)}$. This subgraph has at most $2t(t-1)$ edges. We shall prove that this subgraph has exactly $2t(t-1)$ edges, and thus $t(t-1)$ color repetitions by the definition of $\widetilde{G}^2$. This contradicts the fact that every $K_{2t,t(t-1)}$ receives at most $t(t-1)-1$ color repetitions, and thus completes the proof. Note that every vertex $\vec{b}_{i,j}$ is incident with exactly two edges $\vec{a}_i\vec{b}_{i,j}$ and $\vec{a}_j\vec{b}_{i,j}$ in $\widetilde{H}$. It suffices to show that for any $1\leq i< j\leq t$ and $1\leq k< \ell \leq t$ with $\{i,j\}\neq \{k,\ell\}$, the corresponding edges of $\vec{a}_i\vec{b}_{i,j}$, $\vec{a}_j\vec{b}_{i,j}$, $\vec{a}_k\vec{b}_{k,\ell}$ and $\vec{a}_{\ell}\vec{b}_{k,\ell}$ in $G$ are eight pairwise distinct edges.

If $\{i,j\}\cap \{k, \ell\}\neq \emptyset$, then $\vec{b}_{i,j}$ and $\vec{b}_{k, \ell}$ have a common neighbor. Thus $\vec{b}_{i,j}$ and $\vec{b}_{k, \ell}$ are not equal in any coordinate by Definition~\ref{def:prune energy_graph} (iii). Thus in this case, for each $m\in [2]$, the four edges $a^{(m)}_ib^{(m)}_{i,j}$, $a^{(m)}_jb^{(m)}_{i,j}$, $a^{(m)}_kb^{(m)}_{k,\ell}$ and $a^{(m)}_{\ell}b^{(m)}_{k,\ell}$ are pairwise distinct in $G$. If $\{i,j\}\cap \{k, \ell\}= \emptyset$, then $i, j, k$ and $\ell$ are pairwise distinct. For any $x,y\in \{i,j,k,\ell\}$ with $x\neq y$, the two vertices $\vec{a}_x$ and $\vec{a}_y$ have a common neighbor $\vec{b}_{x,y}$. This implies that for each $m\in [2]$, the four vertices $a^{(m)}_i$, $a^{(m)}_j$, $a^{(m)}_k$ and $a^{(m)}_{\ell}$ are pairwise distinct in $G$. Thus the corresponding edges of $\vec{a}_i\vec{b}_{i,j}$, $\vec{a}_j\vec{b}_{i,j}$, $\vec{a}_k\vec{b}_{k,\ell}$ and $\vec{a}_{\ell}\vec{b}_{k,\ell}$ in $G$ are eight pairwise distinct edges. This completes the proof of Theorem~\ref{th:Kt+}.


\section{Proof of Theorem~\ref{th:C_p}}\label{sec:PfC_p}

To prove Theorem~\ref{th:C_p}, we shall give an upper bound on $|E(\widetilde{G}^2)|$ using the bipartite Tur\'{a}n number for even cycles.

\begin{proof}[Proof of Theorem~\ref{th:C_p}] The case $2\leq p\leq 3$ is trivial since $n^{2-2/\lfloor p/2\rfloor}=1$. The case $4\leq p\leq 5$ holds since $n^{2-2/\lfloor p/2\rfloor}=n$ and $r(K_{n,n}, K_{p,p}, p^2-p+1)\geq r(K_{n,n}, K_{p,p}, p^2-2p+3)=\Omega(n)$ (see \cite[Corollary~4.2]{AxFM}). The case $p=6$ follows from the lower bound given by Axenovich, F\"{u}redi and Mubayi (see \cite[Theorem~4.3]{AxFM}). Thus we may assume that $p\geq 7$. Let $k\colonequals \lfloor p/2\rfloor$.

Let $G=G(A, B)$ be a copy of $K_{n,n}$ with an edge-coloring $c: E(K_{n,n}) \to C(G)$, in which every $K_{p,p}$ receives at least $p^2-p+1$ distinct colors. Equivalently, every $K_{p,p}$ receives at most $p-1$ color repetitions. Let $A=\{a_1, a_2, \ldots, a_n\}$ and $B=\{b_1, b_2, \ldots, b_n\}$. We shall prove that $|C(G)|=\Omega(n^{2-2/\lfloor p/2\rfloor})$. If $|C(G)|=\Omega(n^2/\log n)$, then we are done. So we may assume $|C(G)|=o(n^2/\log n)$.

We first show that $G$ contains no monochromatic $K_{1, p-2}$. For a contradiction, suppose that $\{b_1, a_1, a_2, \ldots, a_{p-2}\}$ forms a monochromatic star. For each $j\in [n]$, let $X_j\colonequals \{i\in C(G)\colon\, c(a_jb_{\ell})=i \mbox{ for some } 2\leq \ell \leq n\}$. Suppose that $|X_j|< (n-1)/3$ for some $j\in [n]$. Then there exist four edges of the same color incident with $a_j$. These four edges together with the monochromatic $K_{1, p-2}$ form a subgraph of $K_{p,p}$ with at least $p-3+3=p$ color repetitions, a contradiction. Hence, we have $|X_j|\geq (n-1)/3$ for every $j\in [n]$. Suppose that $|X_{j_1}\cap X_{j_2}|\geq 3$ for some $1\leq j_1<j_2\leq n$. Recall that $p\geq 7$. Then $G$ contains a subgraph of $K_{p,p}$ with at least $p-3+3=p$ color repetitions, a contradiction. Hence, we have $|X_{j_1}\cap X_{j_2}|\leq 2$ for every $1\leq j_1<j_2\leq n$. By Lemma~\ref{le:Corradi}, we have $|C(G)|\geq |X_1\cup X_2\cup \cdots \cup X_n|\geq \frac{((n-1)/3)^2n}{(n-1)/3+(n-1)2}=\Omega(n^2)$, contradicting our assumption that $|C(G)|=o(n^2/\log n)$. Therefore, $G$ contains no monochromatic $K_{1, p-2}$.

Hence, the pruned second color energy graph $\widetilde{G}^2$ exists. By Theorem~\ref{th:Tu} (iii) and the arguments at the end of Section~\ref{subsec:colorenergy}, it suffices to prove that $\widetilde{G}^2$ is $C_{2k}$-free. Suppose for a contradiction that $\widetilde{G}^2$ contains a copy $\widetilde{C}$ of $C_{2k}$. We write $\widetilde{C}=(a^{(1)}_1, a^{(2)}_1)(b^{(1)}_2, b^{(2)}_2)(a^{(1)}_3, a^{(2)}_3)(b^{(1)}_4,$ $b^{(2)}_4)\cdots (a^{(1)}_{2k-1}, a^{(2)}_{2k-1})(b^{(1)}_{2k}, b^{(2)}_{2k})(a^{(1)}_1, a^{(2)}_1)$, where $a^{(1)}_i\in A_1$, $a^{(2)}_i$ $\in A_2$ for odd $i\in [2k]$, and $b^{(1)}_i\in B_1$, $b^{(2)}_i\in B_2$ for even $i\in [2k]$. By the definition of $\widetilde{G}^2$, we have $c(a^{(1)}_ib^{(1)}_{i+1})=c(a^{(2)}_ib^{(2)}_{i+1})$ (resp., $c(b^{(1)}_ia^{(1)}_{i+1})=c(b^{(2)}_ia^{(2)}_{i+1})$; here $a^{(1)}_{2k+1}\colonequals a^{(1)}_1$ and $a^{(2)}_{2k+1}\colonequals a^{(2)}_1$) for odd $i\in [2k]$ (resp., even $i\in [2k]$). Note that $a^{(1)}_1, a^{(2)}_1, a^{(1)}_3, a^{(2)}_3, \ldots, a^{(1)}_{2k-1}, a^{(2)}_{2k-1}$ (resp., $b^{(1)}_2, b^{(2)}_2, b^{(1)}_4, b^{(2)}_4, \ldots, b^{(1)}_{2k},$ $b^{(2)}_{2k}$) may not be pairwise distinct, and we use $X$ (resp., $Y$) to denote the set of these vertices. So $|X|\leq 2k$ and $|Y|\leq 2k$. In the following, we will show that $G$ contains a subgraph of $K_{p,p}$ with at least $p$ color repetitions, which is a contradiction.

If $|X|=|Y|=2k$, then the subgraph of $G$ induced by $X\cup Y$ is a $K_{2k,2k}$ receiving at least $2k$ color repetitions. This is a contradiction when $p$ is even. When $p$ is odd, there exists an edge of the same color as $a^{(1)}_1b^{(1)}_2$ between $A\setminus X$ and $B\setminus Y$ (this follows from Definition~\ref{def:prune energy_graph} (ii)). Then there exists a $K_{p,p}$ receiving at least $2k+1=p$ color repetitions, a contradiction. Therefore, we have $|X|<2k$ or $|Y|<2k$.

We consider the edges of $\widetilde{C}$ (and the corresponding structures in $G$) one-by-one. In the $i$th step, if $i\in [2k]$ is odd (resp., $i\in [2k]$ is even), then we consider the edge $(a^{(1)}_i, a^{(2)}_i)(b^{(1)}_{i+1}, b^{(2)}_{i+1})$ (resp., $(b^{(1)}_i, b^{(2)}_i)(a^{(1)}_{i+1}, a^{(2)}_{i+1})$). For convenience, let $H_0$ denote a graph with no vertices, and let $A_0=B_0=\emptyset$. For odd $i\in [2k]$, let $A_i \colonequals  A_{i-1}\cup \{a^{(1)}_i, a^{(2)}_i\}$ and $B_i\colonequals B_{i-1}$, and let $H_i$ be the graph obtained from $H_{i-1}$ by adding vertices $a^{(1)}_i, a^{(2)}_i, b^{(1)}_{i+1}, b^{(2)}_{i+1}$ and edges $a^{(1)}_ib^{(1)}_{i+1}$ and $a^{(2)}_ib^{(2)}_{i+1}$ (note that some of these vertices and edges may already be in $H_{i-1}$, and we do not add such vertices and edges repeatedly). For even $i\in [2k]$, let $A_i \colonequals  A_{i-1}$ and $B_i\colonequals B_{i-1}\cup \{b^{(1)}_i, b^{(2)}_i\}$, and let $H_i$ be the graph obtained from $H_{i-1}$ by adding vertices $b^{(1)}_i, b^{(2)}_i, a^{(1)}_{i+1}, a^{(2)}_{i+1}$ and edges $b^{(1)}_ia^{(1)}_{i+1}$ and $b^{(2)}_ia^{(2)}_{i+1}$ (we do not add vertices and edges repeatedly).

For any odd $i\in [2k]$, we have $c(a^{(1)}_ib^{(1)}_{i+1})=c(a^{(2)}_ib^{(2)}_{i+1})$ and one of the following holds:

\begin{itemize}
\item[(1$^o$)] $a^{(1)}_ib^{(1)}_{i+1}\notin E(H_{i-1})$ and $a^{(2)}_ib^{(2)}_{i+1}\notin E(H_{i-1})$;
\item[(2$^o$)] exactly one of the edges $a^{(1)}_ib^{(1)}_{i+1}$ and $a^{(2)}_ib^{(2)}_{i+1}$ is an edge of $H_{i-1}$;
\item[(3$^o$)] $a^{(1)}_ib^{(1)}_{i+1}\in E(H_{i-1})$ and $a^{(2)}_ib^{(2)}_{i+1}\in E(H_{i-1})$.
\end{itemize}

For any even $i\in [2k]$, we have $c(b^{(1)}_ia^{(1)}_{i+1})=c(b^{(2)}_ia^{(2)}_{i+1})$ and one of the following holds:

\begin{itemize}
\item[(1$^e$)] $b^{(1)}_ia^{(1)}_{i+1}\notin E(H_{i-1})$ and $b^{(2)}_ia^{(2)}_{i+1}\notin E(H_{i-1})$;
\item[(2$^e$)] exactly one of the edges $b^{(1)}_ia^{(1)}_{i+1}$ and $b^{(2)}_ia^{(2)}_{i+1}$ is an edge of $H_{i-1}$;
\item[(3$^e$)] $b^{(1)}_ia^{(1)}_{i+1}\in E(H_{i-1})$ and $b^{(2)}_ia^{(2)}_{i+1}\in E(H_{i-1})$.
\end{itemize}

If (1$^o$) or (1$^e$) holds, then we get at least one new color repetition in step $i$.

If (2$^o$) or (2$^e$) holds, then we get exactly one new color repetition in step $i$.

If (3$^o$) (resp., (3$^e$)) holds, then by the definition of $\widetilde{G}^2$, we have $a^{(1)}_i, a^{(2)}_i\in A_{i-2}$ and $b^{(1)}_{i+1}, b^{(2)}_{i+1}\in B_{i-1}$ (resp., $b^{(1)}_i, b^{(2)}_i\in B_{i-2}$ and $a^{(1)}_{i+1}, a^{(2)}_{i+1}\in A_{i-1}$).

Let $m_o$ and $m_e$ be the number of steps in which (3$^o$) and (3$^e$) applies, respectively. Then $|A_{2k}|\leq 2k-2m_o$ and $|B_{2k}|\leq 2k-2m_e$. On the other hand, if (3$^o$) applies in step $i$ for some odd $i\in [2k]$, then we also have $B_{i+1}=B_{i-1}$. Thus $|B_{2k}|\leq 2k-2m_o$, so $|B_{2k}|\leq 2k-2\max\{m_o, m_e\}\leq 2k-m_o-m_e$. If (3$^e$) applies in step $i$ for some even $i\in [2k-2]$, then we also have $A_{i+1}=A_{i-1}$. Thus $|A_{2k}|\leq 2k-2(m_e-1)$, so $|A_{2k}|\leq 2k-2\max\{m_o, m_e-1\}\leq 2k-m_o-m_e+1$.

Note that $H_{2k}$ (as a subgraph of $G$) has at least $2k-m_o-m_e$ color repetitions. Let $c_0$ be the color used on $a^{(1)}_1b^{(1)}_2$. By Definition~\ref{def:prune energy_graph} (ii), color $c_0$ appears at least $\log n$ times in $G$.

We claim that $|A_{2k}|\geq 2k-m_o-m_e+1$ (and thus $|A_{2k}|= 2k-m_o-m_e+1$). Otherwise if $|A_{2k}|\leq 2k-m_o-m_e$, then we can add $p-(2k-m_o-m_e)$ additional edges of color $c_0$ (and at most $p-(2k-m_o-m_e)$ vertices from each part of $G$ if necessary) to $H_{2k}$, forming a subgraph of $K_{p,p}$ with at least $p$ color repetitions. This contradicts the assumption that every $K_{p,p}$ receives at least $p^2-p+1$ distinct colors in $G$.

If $m_o\neq m_e-1$, then it is easy to check that $|A_{2k}|\leq 2k-2\max\{m_o, m_e-1\}\leq 2k-m_o-m_e$, a contradiction. If (3$^e$) does not apply in step $2k$, then $|A_{2k}|\leq 2k-2\max\{m_o, m_e\}\leq 2k-m_o-m_e$, a contradiction. Hence, we have $m_o=m_e-1$ and (3$^e$) applies in step $2k$.

Since (3$^e$) applies in step $2k$, we may assume that $b^{(1)}_{2k}a^{(1)}_1\in E(H_{j_1})\setminus E(H_{j_1-1})$ and $b^{(2)}_{2k}a^{(2)}_1\in E(H_{j_2})\setminus E(H_{j_2-1})$ for some $j_1, j_2\in [2k-1]$. Note that $j_1\neq j_2$, since otherwise we get repeated vertices in the cycle $\widetilde{C}$. Without loss of generality, let $j_1<j_2$. Let $c'=c(b^{(1)}_{2k}a^{(1)}_1)=c(b^{(2)}_{2k}a^{(2)}_1)$. Let $xy$ be the other edge in step $j_2$ with $x\in A_1$ and $y\in B_1$. Then $c(xy)=c'$. Let $\ell$ be the minimum value such that there is an edge of color $c'$ in $H_{\ell}$. So $\ell\leq j_1<j_2<2k$.

If $xy\notin E(H_{j_2-1})$, then we get two new color repetitions in step $j_2$. Thus $H_{2k}$ has at least $2k-m_o-m_e+1$ color repetitions. We can add $p-(2k-m_o-m_e+1)$ additional edges of color $c_0$ (and some additional vertices if necessary) to $H_{2k}$, forming a subgraph of $K_{p,p}$ with at least $p$ color repetitions, a contradiction.

If $xy\in E(H_{j_2-1})$ and $j_2$ is odd, then (2$^o$) applies in step $j_2$. Since $b^{(2)}_{2k}a^{(2)}_1\in E(H_{j_2})\setminus E(H_{j_2-1})$, we have $a^{(2)}_1=a^{(2)}_{j_2}$. Thus $|A_{j_2}\setminus A_{j_2-2}|\leq 1$. Then $|A_{2k}|\leq 2k-2m_o-1=2k-m_o-m_e$, a contradiction.

Hence, $xy\in E(H_{j_2-1})$ and $j_2$ is even. We first claim that for any even $i\in [2k-2]$, if we have (3$^e$) in step $i$, then we also have (3$^o$) in step $i+1$. Indeed, if we have (3$^e$) in step $i$ but not (3$^o$) in step $i+1$, then $A_{i+1}=A_{i-1}$ and $|A_{2k}|\leq 2k-2m_o-2=2k-m_o-m_e-1$. This is a contradiction. Moreover, recall that $m_o=m_e-1$ and (3$^e$) applies in step $2k$. Thus (3$^e$) applies in step $i$ if and only if (3$^o$) applies in step $i+1$ for any even $i\in [2k-2]$. Since $b^{(2)}_{2k}a^{(2)}_1\notin E(H_{j_2-1})$, (3$^e$) does not  apply in step $j_2$. Hence, (3$^o$) does not apply in step $j_2+1$. Since $b^{(2)}_{2k}a^{(2)}_1\in E(H_{j_2})\setminus E(H_{j_2-1})$, we have  $a^{(2)}_1=a^{(2)}_{j_2+1}$. Thus $|A_{j_2+1}\setminus A_{j_2-1}|\leq 1$. Then $|A_{2k}|\leq 2k-2m_o-1=2k-m_o-m_e$. This contradiction completes the proof.
\end{proof}

Let $1\leq a\leq s$, $1\leq b\leq t$ and $0\leq m\leq ab-2$. It is easy to see that $r(K_{n,n}, K_{a,b}, ab-m)\leq r(K_{n,n}, K_{s,t}, st-m)$. Combining this fact and Theorem~\ref{th:C_p}, we obtain the following result.

\begin{corollary}\label{co:Kss+1} For any integer $p\geq 2$, we have $r(K_{n,n}, K_{p,p+1}, p(p+1)-p+1)=\Omega\left(n^{2-\frac{2}{\lfloor p/2\rfloor}}\right)$.
\end{corollary}


\section{An advanced technique}\label{sec:adv}

In the above proof of Theorem~\ref{th:C_p}, our extension of the Color Energy Method enables us to consider the edges of a $(2k)$-cycle $\widetilde{C}$ one-by-one, and use the existence of additional edges with the same color to confirm the presence of a desired $K_{p,p}$. Unfortunately, this technique is not advanced enough to prove Theorems~\ref{th:r=2theta} and \ref{th:r3theta}. However, we can build on the recently developed enhanced version of the Color Energy Method due to Balogh, English, Heath and Krueger \cite{BEHK}. In \cite{BEHK}, they developed a framework for studying $f(n,p,q)$. In this section, we modify their framework and show how we can use it to prove results on $r(K_{n,n}, K_{s,t}, q)$, in particular Theorems~\ref{th:r=2theta} and \ref{th:r3theta}.

Let $G=G(A, B)$ be a copy of $K_{n,n}$ with an edge-coloring $c\colon\, E(K_{n,n})\to C(G)$. Assume that the pruned $r$th color energy graph $\widetilde{G}^r$ of $G$ exists. For every $k\in [r]$, we define the \emph{$k$th coordinate map} $\pi_k$ as follows. If $\vec{x}=(x_1, x_2, \ldots, x_r)$ is a vertex of $\widetilde{G}^r$, then $\pi_k(\vec{x})\colonequals x_k$. If $\vec{x}\vec{y}$ is an edge of $\widetilde{G}^r$, then $\pi_k(\vec{x}\vec{y})\colonequals \pi_k(\vec{x})\pi_k(\vec{y})$, hence $\pi_k(\vec{x}\vec{y})$ is an edge of $G$. If $\widetilde{V}$ is a subset of $V(\widetilde{G}^r)$, then $\pi_k(\widetilde{V})\colonequals \{\pi_k(\vec{x})\colon\, \vec{x}\in \widetilde{V}\}$. If $\widetilde{G}$ is a subgraph of $\widetilde{G}^r$, then $\pi_k(\widetilde{G})$ is a subgraph of $G$ with vertex set $\pi_k(V(\widetilde{G}))$ and edge set $\{\pi_k(\vec{e})\colon\, \vec{e}\in E(\widetilde{G})\}$. Moreover, for any structure $\sigma$ in $\widetilde{G}^r$ (where $\sigma$ could indicate a vertex, an edge, a vertex set or a subgraph), let $\pi(\sigma)\colonequals \bigcup_{k\in [r]}\pi_k(\sigma)$. Finally, for any subgraph $F\subseteq G$ and any structure $\sigma$ in $\widetilde{G}^r$, let $F \cup \pi(\sigma)$ be the subgraph of $G$ obtained from $F$ by adding $\pi(\sigma)$ (here we note again that some vertices or edges of $\pi(\sigma)$ may already be in $F$, and we do not add such vertices or edges repeatedly).

Let $H\subseteq G$ and $\widetilde{G}\subseteq \widetilde{G}^r$ such that $\widetilde{G}$ contains no isolated vertices and $|E(\widetilde{G})|=m$. An ordering $\vec{u}_1\vec{v}_1, \vec{u}_2\vec{v}_2, \ldots, \vec{u}_m\vec{v}_m$ of $E(\widetilde{G})$ is called \emph{$H$-compatible} if $\pi(\vec{u}_1)\subseteq V(H)$ and $\pi(\vec{u}_i)\subseteq V(H\cup (\bigcup^{i-1}_{j=1}\pi(\vec{u}_j\vec{v}_j)))$ for each $i\in \{2, \ldots, m\}$. Let $H_0, H_1, H_2, \ldots, H_m$ be a sequence of graphs such that $H_0\colonequals H$ and $H_i\colonequals H_{i-1}\cup \pi(\vec{u}_i\vec{v}_i)$ for each $i\in [m]$. Note that $H_m= H\cup \pi(\widetilde{G})$. Recall that $\widetilde{G}^r$ is a bipartite graph with partite sets $A_1\times \cdots \times A_r$ and $B_1\times \cdots \times B_r$. Let $I_A\colonequals \{i\in [m]\colon\, \vec{v}_i\in A_1\times \cdots \times A_r\}$ and $I_B\colonequals \{i\in [m]\colon\, \vec{v}_i\in B_1\times \cdots \times B_r\}$. Let $m_A\colonequals |I_A|$ and $m_B\colonequals |I_B|$. Note that $m_A+m_B=m$. For each $i\in [m]$ and $k\in [r]$, let
\begin{itemize}
\item $n_{i,k}\colonequals 1$ if $\pi_k(\vec{v}_i)\notin V(H_{i-1})$, and $n_{i,k}\colonequals 0$ otherwise;
\item $s_{i,k}\colonequals 1$ if $\pi_k(\vec{v}_i)\in V(H_{i-1})$ but $\pi_k(\vec{u}_i\vec{v}_i)\notin E(H_{i-1})$, and $s_{i,k}\colonequals 0$ otherwise;
\item $d_{i,k}\colonequals 1$ if $\pi_k(\vec{u}_i\vec{v}_i)\in E(H_{i-1})$, and $d_{i,k}\colonequals 0$ otherwise.
\end{itemize}
Note that $n_{i,k}+s_{i,k}+d_{i,k}=1$ for every $i\in [m]$ and $k\in [r]$. Moreover, let $n_i\colonequals \sum^{r}_{k=1}n_{i,k}$, $s_{i}\colonequals \sum^{r}_{k=1}s_{i,k}$ and $d_i\colonequals \sum^{r}_{k=1}d_{i,k}$. Note that $n_i+s_i+d_i=r$ for every $i\in [m]$. Finally, let $N_A\colonequals \sum_{i\in I_A}n_i$, $N_B\colonequals \sum_{i\in I_B}n_i$, $S_A\colonequals \sum_{i\in I_A}s_i$, $S_B\colonequals \sum_{i\in I_B}s_i$, $D_A\colonequals \sum_{i\in I_A}d_i$ and $D_B\colonequals \sum_{i\in I_B}d_i$. Note that $N_A+S_A+D_A+N_B+S_B+D_B=rm_A+rm_B=rm$.

Let $F\subseteq G$, $\widetilde{R}_A\subseteq A_1\times \cdots \times A_r$, $\widetilde{R}_B\subseteq B_1\times \cdots \times B_r$, $\vec{a}\in A_1\times \cdots \times A_r$ and $\vec{b}\in B_1\times \cdots \times B_r$. The 4-tuple $(\widetilde{R}_A, \widetilde{R}_B, \vec{a}, \vec{b})$ is called an \emph{$F$-reservoir with sources $\vec{a}$ and $\vec{b}$} if the following holds:
\begin{itemize}
\item $\pi(\vec{a}), \pi(\vec{b})\subseteq V(F)$;
\item $\vec{a}\vec{x}\in E(\widetilde{G}^r)$ for all $\vec{x}\in \widetilde{R}_B$, and $\vec{b}\vec{y}\in E(\widetilde{G}^r)$ for all $\vec{y}\in \widetilde{R}_A$;
\item $\pi(\widetilde{R}_A\cup \widetilde{R}_B)\cap V(F)=\emptyset$, and $\pi(\vec{x})\cap \pi(\vec{y})=\emptyset$ for any two distinct vertices $\vec{x}, \vec{y}\in \widetilde{R}_A\cup \widetilde{R}_B$.
\end{itemize}
Based on the above considerations, the following key lemma provides us with a tool to demonstrate the existence of a graph with a sufficient number of color repetitions.

\begin{lemma}\label{le:reservoir} Let $F\subseteq G$ with an $F$-reservoir $(\widetilde{R}_A, \widetilde{R}_B, \vec{a}, \vec{b})$, and let $D_1, D_2$ be two nonnegative integers. If $D_1\leq r|\widetilde{R}_A|$ and $D_2\leq r|\widetilde{R}_B|$, then there exists a graph $F^{\ast}\supseteq F$ such that
\begin{itemize}
\item[{\rm (i)}] $V(F^{\ast})\subseteq V(F)\cup \pi(\widetilde{R}_A\cup \widetilde{R}_B)$;
\item[{\rm (ii)}] $|V(F^{\ast})\cap A|=|V(F)\cap A|+D_1$, $|V(F^{\ast})\cap B|=|V(F)\cap B|+D_2$;
\item[{\rm (iii)}] $F^{\ast}$ has at least $\lfloor D_1(r-1)/r\rfloor+\lfloor D_2(r-1)/r\rfloor$ more color repetitions than $F$.
\end{itemize}
\end{lemma}

\begin{proof} For each $j\in \{1, 2\}$, let $w_j, z_j$ be integers with $0\leq z_j\leq r-1$ such that $D_j=w_jr+z_j$. Note that $\lfloor D_j(r-1)/r\rfloor= \lfloor w_jr(r-1)/r+ z_j(r-1)/r\rfloor= w_j(r-1)+z_j-\lceil z_j/r\rceil= w_j(r-1)+z_j-\mathds{1}_{(z_j\neq 0)}$, where $\mathds{1}_{(z_j\neq 0)}\colonequals 1$ if $z_j\neq 0$, and $\mathds{1}_{(z_j\neq 0)}\colonequals 0$ otherwise. Let $w_1'=w_1+\mathds{1}_{(z_1\neq 0)}$ and $w_2'=w_2+\mathds{1}_{(z_2\neq 0)}$. We first choose $w_1'$ vertices $\vec{y}_1, \ldots, \vec{y}_{w_1'}$ from $\widetilde{R}_A$. We form a graph $F'$ by adding $\big(\bigcup_{1\leq \ell \leq w_1}\pi(\vec{y}_{\ell}\vec{b})\big)\cup \big(\bigcup_{1\leq k\leq z_1}\pi_k(\vec{y}_{w_1'}\vec{b})\big)$ to $F$. We next choose $w_2'$ vertices $\vec{x}_1, \ldots, \vec{x}_{w_2'}$ from $\widetilde{R}_B$. We form $F^{\ast}$ by adding $\big(\bigcup_{1\leq \ell \leq w_2}\pi(\vec{x}_{\ell}\vec{a})\big)\cup \big(\bigcup_{1\leq k\leq z_2}\pi_k(\vec{x}_{w_2'}\vec{a})\big)$ to $F'$. It is easy to check that $F^{\ast}$ satisfies the above three requirements.
\end{proof}

In the next subsection, we show how to prove Theorems~\ref{th:r=2theta} and \ref{th:r3theta} using the above framework, in combination with our generalization of Corr\'{a}di's Lemma and some known Tur\'{a}n numbers of theta graphs. We believe our technique may turn out to be useful in proving other results, e.g., by combining it with other (bipartite) Tur\'{a}n numbers.


\subsection{Proofs of Theorems~\ref{th:r=2theta} and \ref{th:r3theta}}

We first present our proof of Theorem~\ref{th:r3theta}.

\begin{proof}[Proof of Theorem~\ref{th:r3theta}]
Let $G=G(A,B)$ be an edge-colored $K_{n,n}$ such that every $K_{p,p}$ receives at least $p^2-(r-1)st+2$ colors (i.e., every $K_{p,p}$ receives at most $(r-1)st-2$ color repetitions). Let $A=\{a_1, a_2, \ldots, a_n\}$ and $B=\{b_1, b_2, \ldots, b_n\}$. Our goal is to prove $|C(G)|=\Omega\left(n^{\frac{r}{r-1}\left(1-\frac{1}{s}\right)}\right)$. So we may assume $|C(G)|=o\left(n^{\frac{r}{r-1}\left(1-\frac{1}{s}\right)}\right)$; otherwise we are done.

As in our earlier proofs, we first show that $G$ contains no monochromatic star, in this case no monochromatic
$K_{1,p-r}$.

\begin{claim}\label{cl:r3theta-1} $G$ contains no monochromatic copy of $K_{1,p-r}$.
\end{claim}

\begin{proof} Suppose for a contradiction that $G$ contains a monochromatic copy of $K_{1, p-r}$, say with vertex set $\{b_1, a_1, \ldots, a_{p-r}\}$. Note that such a $K_{1,p-r}$ is a subgraph of $K_{p,p}$ with $p-r-1$ color repetitions. For any $j\in [n]$, let $X_j\colonequals \{i\in C(G)\colon\, c(a_jb_{\ell})=i \mbox{ for some } 2\leq \ell \leq n\}$. We first show that $|X_j|\geq (n-1)/((r-1)st-p+r)$ for every $j\in [n]$. Otherwise there exist $(r-1)st-p+r+1$ edges of the same color between $a_j$ and $\{b_2, \ldots, b_n\}$ for some $j\in [n]$. Since $(r-1)st-p+r+2\leq p$ and $(r-1)st-p+r+p-r-1=(r-1)st-1$, in this case there is a $K_{p,p}$ receiving at least $(r-1)st-1$ color repetitions, a contradiction. We next show that $|X_{j_1}\cap \cdots \cap X_{j_r}|< \lceil h\rceil$ for every $1\leq j_1< \cdots < j_r\leq n$, where $h\colonequals ((r-1)st-p+r)/(r-1)$. Otherwise, suppose that there exist some $1\leq j_1< \cdots < j_r\leq n$ such that $C(a_{j_{1}}, B_{j_{1}})=\cdots =C(a_{j_{r}}, B_{j_{r}})$, where $B_{j_1}, \ldots, B_{j_r}\subseteq \{b_2, \ldots, b_n\}$ and $|B_{j_1}|=\cdots =|B_{j_r}|=\lceil h\rceil$. Since
\begin{align*}
 |B_{j_1}\cup \cdots \cup B_{j_r}|-p\leq &~r\lceil h\rceil-p =~r\lc\frac{(r-1)st-p+r}{r-1}\rc-p\leq~r\frac{(r-1)st-p+2r-2}{r-1}-p \\
 = &~\frac{r(2(r-1)st-(r(s-1)t+2r)+4r-4-(r-1)((s-1)t+2))}{2(r-1)} \\
 = &~\frac{r(-st+(2r-1)t-2)}{2(r-1)}\leq~\frac{-r}{r-1}\leq~-1,
\end{align*}
we have $|B_{j_1}\cup \cdots \cup B_{j_r}\cup \{b_1\}|\leq p$. Since $p-r-1+(r-1)\lceil h\rceil\geq (r-1)st-1$, there is a $K_{p,p}$ receiving at least $(r-1)st-1$ color repetitions, a contradiction. Since $|C(G)|\geq |X_1\cup X_2\cup \cdots \cup X_n|$ and by Lemma~\ref{le:gCorradi}, we have that $|C(G)|$ is greater than
\begin{align*}
 \left(\frac{((n-1)/((r-1)st-p+r))^rn^{r-1}}{((n-1)/((r-1)st-p+r))(n^{r-1}-(n-1)!/(n-r)!)+((n-1)!/(n-r)!)\lceil h\rceil}\right)^{1/(r-1)},
\end{align*}
which is $\Omega\left(n^{\frac{r}{r-1}}\right)$. This contradicts the assumption that $|C(G)|=o\left(n^{\frac{r}{r-1}\left(1-\frac{1}{s}\right)}\right)$.
\end{proof}

By the above claim, the pruned $r$th color energy graph $\widetilde{G}^r$ of $G$ exists. By Theorem~\ref{th:Tu} (ii), we have ex$(n^r, \mit\Theta$$(s, 2rs^2t))=O(n^{r(1+1/s)})$. By the lower bound (\ref{eq:color_edge}), it suffices to prove that $\widetilde{G}^r$ contains no $\mit\Theta$$(s, 2rs^2t)$. Suppose for a contradiction that $\widetilde{G}^r$ contains a copy $\widetilde{\mit\Theta}$ of $\mit\Theta$$(s, 2rs^2t)$. Since $s$ is odd, the two vertices of degree $2rs^2t$ in $\widetilde{\mit\Theta}$ are contained in distinct partite sets of the bipartition of $\widetilde{G}^r$. Let $\vec{a}\in A_1\times \cdots \times A_r$ and $\vec{b}\in B_1\times \cdots \times B_r$ be the two vertices of degree $2rs^2t$ in $\widetilde{\mit\Theta}$. Let $\widetilde{P}_1, \widetilde{P}_2, \ldots, \widetilde{P}_{2rs^2t}$ be the $2rs^2t$ paths connecting $\vec{a}$ and $\vec{b}$ in $\widetilde{\mit\Theta}$, where $\widetilde{P}_j\colonequals \vec{a}^j_0 \vec{b}^j_1 \vec{a}^j_2 \vec{b}^j_3 \cdots \vec{a}^j_{s-1} \vec{b}^j_{s}$ for each $j\in [2rs^2t]$ (here $\vec{a}^j_0\colonequals \vec{a}$ and $\vec{b}^j_s\colonequals \vec{b}$).

We first choose a sequence of $t+2st$ distinct paths $\widetilde{P}_{j_1}, \widetilde{P}_{j_2}, \ldots, \widetilde{P}_{j_{t+2st}}$ such that $\vec{b}^{j_{\alpha}}_{1}$ has no common coordinates with any vertices in $X\colonequals \{\vec{a}, \vec{b}\}\cup (\bigcup_{1\leq \ell \leq \alpha-1} V(\widetilde{P}_{j_{\ell}}))$ for all $1\leq \alpha\leq t+st$, and $\vec{a}^{j_{\beta}}_{s-1}$ has no common coordinates with any vertices in $Y\colonequals \{\vec{a}, \vec{b}\}\cup (\bigcup_{1\leq \ell \leq t} V(\widetilde{P}_{j_{\ell}}))\cup (\bigcup_{t+1\leq \ell \leq t+st}\{\vec{b}^{j_{\ell}}_{1}\}) \cup (\bigcup_{t+st+1\leq \ell \leq \beta-1}\{\vec{a}^{j_{\ell}}_{s-1}\})$ for all $t+st+1\leq \beta \leq t+2st$. We first show that such a sequence of paths exists. Note that $|X|\leq 2+(s-1)(t+st-1)< s^2t$ and $|Y|\leq 2+(s-1)t+st+st-1< s^2t$. Thus $X$ (resp., $Y$) results in fewer than $rs^2t$ possible coordinate conflicts with possible choices of $\vec{b}^{j_{\alpha}}_1$ (resp., $\vec{a}^{j_{\beta}}_{s-1}$). By Definition~\ref{def:prune energy_graph} (iii), each possible coordinate conflict resulting from $X$ (resp., $Y$) removes at most one choice for $\vec{b}^{j_{\alpha}}_1$ (resp., $\vec{a}^{j_{\beta}}_{s-1}$). Since there are $2rs^2t$ paths in total, we can find a desired sequence of paths. Without loss of generality, we may assume that $\widetilde{P}_1, \widetilde{P}_2, \ldots, \widetilde{P}_{t+2st}$ is such a sequence of paths.

Let $H$ be a subgraph of $G$ with $V(H)=\pi(\vec{a})\cup \pi(\vec{b})$ and $E(H)=\emptyset$. Let $\widetilde{R}_A=\{\vec{a}^{t+st+1}_{s-1}, \ldots, \vec{a}^{t+2st}_{s-1}\}$, $\widetilde{R}_B=\{\vec{b}^{t+1}_{1}, \ldots,$ $\vec{b}^{t+st}_{1}\}$ and $\widetilde{G}$ be the graph formed by $\widetilde{P}_1, \widetilde{P}_2, \ldots, \widetilde{P}_t$. Note that $\widetilde{G}$ is in fact a $\mit\Theta$$(s,t)$ in $\widetilde{G}^r$, and $(\widetilde{R}_A, \widetilde{R}_B, \vec{a}, \vec{b})$ is an $(H\cup \pi(\widetilde{G}))$-reservoir (see Figure~\ref{fig:reservoir}). Let $m\colonequals st$. In the following, we will use $\vec{u}_1\vec{v}_1, \vec{u}_2\vec{v}_2, \ldots, \vec{u}_m\vec{v}_m$ to denote the edges $\vec{a}\vec{b}^1_1, \vec{b}^1_1\vec{a}^1_2, \ldots, \vec{a}^1_{s-1}\vec{b}, \vec{a}\vec{b}^2_1, \vec{b}^2_1\vec{a}^2_2, \ldots,\vec{a}^2_{s-1}\vec{b}, \ldots, \vec{a}\vec{b}^t_1, \vec{b}^t_1\vec{a}^t_2, \ldots, \vec{a}^t_{s-1}\vec{b}$, respectively. Then $E(\widetilde{G})=\{\vec{u}_1\vec{v}_1, \vec{u}_2\vec{v}_2, \ldots, \vec{u}_m\vec{v}_m\}$ and this is an $H$-compatible ordering of $E(\widetilde{G})$.

\begin{figure}[htbp]
\begin{center}
\begin{tikzpicture}[scale=0.06,auto,swap]
\tikzstyle{vertex}=[circle,draw=black,fill=black]
\node[vertex,scale=0.5] (a) at (-60,0) {}; \draw (-60,6) node {$\vec{a}$};
\node[vertex,scale=0.5] (a12) at (-26,0) {}; \draw (-26,6) node {$\vec{a}^{1}_{2}$};
\node[vertex,scale=0.5] (a22) at (-13,0) {}; \draw (-13,6) node {$\vec{a}^{2}_{2}$};
\node[vertex,scale=0.5] (at2) at (2,0) {}; \draw (2,6) node {$\vec{a}^{t}_{2}$};
\node[vertex,scale=0.5] (a1s-1) at (32,0) {}; \draw (32,6) node {$\vec{a}^{1}_{s-1}$};
\node[vertex,scale=0.5] (a2s-1) at (45,0) {}; \draw (45,6) node {$\vec{a}^{2}_{s-1}$};
\node[vertex,scale=0.5] (ats-1) at (60,0) {}; \draw (60,6) node {$\vec{a}^{t}_{s-1}$};

\node[vertex,scale=0.5] (b11) at (-60,-50) {}; \draw (-60,-56) node {$\vec{b}^{1}_{1}$};
\node[vertex,scale=0.5] (b21) at (-47,-50) {}; \draw (-47,-56) node {$\vec{b}^{2}_{1}$};
\node[vertex,scale=0.5] (bt1) at (-32,-50) {}; \draw (-32,-56) node {$\vec{b}^{t}_{1}$};
\node[vertex,scale=0.5] (b13) at (-16,-50) {}; \draw (-16,-56) node {$\vec{b}^{1}_{3}$};
\node[vertex,scale=0.5] (b23) at (-3,-50) {}; \draw (-3,-56) node {$\vec{b}^{2}_{3}$};
\node[vertex,scale=0.5] (bt3) at (12,-50) {}; \draw (12,-56) node {$\vec{b}^{t}_{3}$};
\node[vertex,scale=0.5] (b) at (60,-50) {}; \draw (60,-56) node {$\vec{b}$};

\draw[dotted,thick] (-10,0)--(-1,0); \draw[dotted,thick] (48,0)--(57,0);
\draw[dotted,thick] (-44,-50)--(-35,-50); \draw[dotted,thick] (0,-50)--(9,-50);
\draw[dotted,thick] (10,0)--(24,0); \draw[dotted,thick] (20,-50)--(52,-50);

\draw[thick] (a)--(b11); \draw[thick] (a)--(b21); \draw[thick] (a)--(bt1);
\draw[thick] (b11)--(a12); \draw[thick] (b21)--(a22); \draw[thick] (bt1)--(at2);
\draw[thick] (a12)--(b13); \draw[thick] (a22)--(b23); \draw[thick] (at2)--(bt3);
\draw[thick] (a1s-1)--(b); \draw[thick] (a2s-1)--(b); \draw[thick] (ats-1)--(b);

\draw (-82,-50) circle [x radius=14cm, y radius=50mm]; \draw (-82,-51) node {$\widetilde{R}_B$};
\draw (82,0) circle [x radius=14cm,y radius=50mm]; \draw (82,1) node {$\widetilde{R}_A$};
\draw[thick] (a)--(-90,-47); \draw[thick] (a)--(-86,-47); \draw[thick] (a)--(-82,-47); \draw[thick] (a)--(-78,-47); \draw[thick] (a)--(-74,-47);
\draw[thick] (b)--(74,-3); \draw[thick] (b)--(78,-3); \draw[thick] (b)--(82,-3); \draw[thick] (b)--(86,-3); \draw[thick] (b)--(90,-3);

\draw (0,0) circle [x radius=110cm, y radius=110mm];
\draw (0,-50) circle [x radius=110cm, y radius=110mm];
\end{tikzpicture}
\caption{The graph $\mit\Theta$$(s,t)$ and the $(H\cup \pi(\widetilde{G}))$-reservoir $(\widetilde{R}_A, \widetilde{R}_B, \vec{a}, \vec{b})$.}
\label{fig:reservoir}
\end{center}
\end{figure}
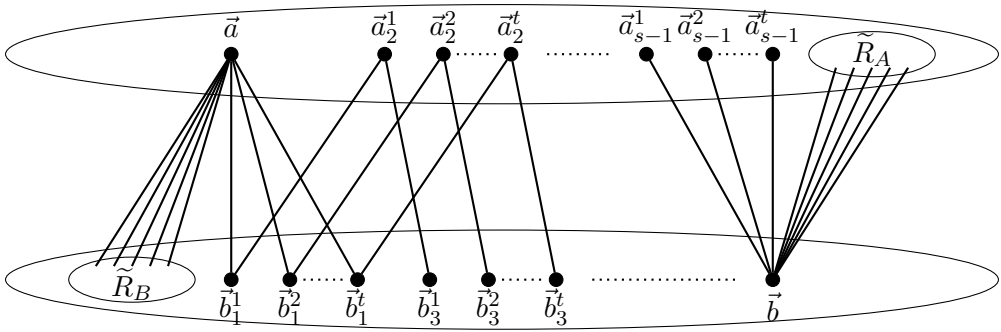

In the following, we will use the notation $I_A, I_B, m_A, m_B, n_i, s_i, d_i, N_A, N_B, S_A, S_B, D_A, D_B$ as it was introduced before Lemma~\ref{le:reservoir}. Let $F\colonequals H\cup \pi(\widetilde{G})$ and $I_B'\colonequals \{s, 2s, \ldots, ts\}\subseteq I_B$. Recall that $N_A+S_A+D_A=rm_A$ and $N_B+S_B+D_B=rm_B$. Moreover, since $\pi(\vec{b})\subseteq V(H)$, we have $n_i=0$ (and thus $s_i+d_i=r$) for all $i\in I_B'$. Thus
\begin{align}\label{al:r3theta-1}
 |V(F)\cap A|= &~|V(H)\cap A|+N_A=~|\pi(\vec{a})|+rm_A-S_A-D_A \nonumber\\
 = &~r+rt(s-1)/2-S_A-D_A=~p-S_A-D_A,
\end{align}
and
\begin{align}\label{al:r3theta-2}
 |V(F)\cap B|= &~|V(H)\cap B|+N_B=~|\pi(\vec{b})|+rm_B-S_B-D_B \nonumber\\
 = &~r+rt(s+1)/2-S_B-D_B=~p+rt-S_B-D_B \nonumber\\
 = &~p-S_B-D_B+\sum_{i\in I_B'}s_i +\sum_{i\in I_B'}d_i.
\end{align}
Furthermore, $F$ has at least
\begin{align}\label{al:r3theta-3}
 \sum_{i\in [m]}(n_i+s_i-\mathds{1}_{(d_i=0)})= &~N_A+N_B+S_A+S_B-\sum_{i\in [m]}\mathds{1}_{(d_i=0)} \nonumber\\
 = &~rst-D_A-D_B-\sum_{i\in [m]}\mathds{1}_{(d_i=0)}
\end{align}
color repetitions.

Let $\gamma\colonequals  \min\{\sum_{i\in I_B'}d_i, S_A\}$, $D_1\colonequals  D_A+\gamma$ and $D_2\colonequals  D_B-\gamma$. Note that $D_1\leq D_A+S_A\leq rm_A\leq rst= r|\widetilde{R}_A|$ and $D_2\leq D_B\leq rm_B\leq rst= r|\widetilde{R}_B|$. By Lemma~\ref{le:reservoir} and equalities~(\ref{al:r3theta-1}), (\ref{al:r3theta-2}) and (\ref{al:r3theta-3}), there exists a graph $F^{\ast}$ such that
\begin{align*}
|V(F^{\ast})\cap A|=~|V(F)\cap A|+D_1=~p-S_A-D_A+D_A+\gamma\leq p,
\end{align*}
\begin{align*}
|V(F^{\ast})\cap B|= &~|V(F)\cap B|+D_2=~p-S_B-D_B+\sum_{i\in I_B'}s_i +\sum_{i\in I_B'}d_i+D_B-\gamma \\
= &~p-\bigg(S_B-\sum_{i\in I_B'}s_i\bigg) +\sum_{i\in I_B'}d_i-\gamma,
\end{align*}
and $F^{\ast}$ has at least
\begin{align*}
~ &~rst-D_A-D_B-\sum_{i\in [m]}\mathds{1}_{(d_i=0)}+\lfloor D_1(r-1)/r\rfloor+\lfloor D_2(r-1)/r\rfloor \\
= &~rst-D_A-D_B-\sum_{i\in [m]}\mathds{1}_{(d_i=0)}+D_1-\lceil D_1/r\rceil+D_2-\lceil D_2/r\rceil \\
\geq &~rst-\sum_{i\in [m]}\mathds{1}_{(d_i=0)}-\lceil D_1/r+D_2/r\rceil-1 \\
= &~rst-\bigg(m-\sum_{i\in [m]}\mathds{1}_{(d_i\neq 0)}\bigg)-\lceil(D_A+D_B)/r\rceil-1 \\
\geq &~ rst-m-1 =~(r-1)st-1
\end{align*}
color repetitions.

Finally, we show that $|V(F^{\ast})\cap B|\leq p$, which implies that $F^{\ast}$ is a subgraph of $K_{p,p}$ receiving at least $(r-1)st-1$ color repetitions. This contradiction completes the proof. If $\sum_{i\in I_B'}d_i\leq S_A$, then $\gamma= \sum_{i\in I_B'}d_i$. Thus $|V(F^{\ast})\cap B|= p-(S_B-\sum_{i\in I_B'}s_i) +\sum_{i\in I_B'}d_i-\gamma= p-(S_B-\sum_{i\in I_B'}s_i)\leq p$. Hence, it remains to consider the case $\sum_{i\in I_B'}d_i> S_A$.

We first prove the following claim.

\begin{claim}\label{cl:r3theta-2} $S_A+S_B\geq rt$.
\end{claim}

\begin{proof} Recall that for each $j\in [t]$, $\vec{b}^{j}_{1}$ has no common coordinates with any vertices in $\{\vec{a}, \vec{b}\}\cup (\bigcup_{1\leq \ell \leq j-1} V(\widetilde{P}_{\ell}))$. Thus for each $j\in [t]$ and $k\in [r]$, we have $\pi_k(\vec{b}^j_1)\notin V(H\cup (\bigcup_{1\leq \ell \leq j-1}\pi(\widetilde{P}_{\ell})))$. Let $g\in [s]$ be the smallest index such that one of $\pi_k(\vec{a}^j_g)\in V(H\cup (\bigcup_{1\leq \ell \leq j-1}\pi(\widetilde{P}_{\ell}))) \cup \{\pi(\vec{b}^j_1), \pi(\vec{a}^j_2), \ldots, \pi(\vec{b}^j_{g-1})\}$ and $\pi_k(\vec{b}^j_g)\in V(H\cup (\bigcup_{1\leq \ell \leq j-1}\pi(\widetilde{P}_{\ell}))) \cup \{\pi(\vec{b}^j_1), \pi(\vec{a}^j_2), \ldots, \pi(\vec{a}^j_{g-1})\}$ holds. Note that such index $g$ exists since $\pi_k(\vec{b}^j_s)= \pi_k(\vec{b})\in V(H)$. Assume that such $g$ appears in the $i$th edge in the ordering $\vec{u}_1\vec{v}_1, \vec{u}_2\vec{v}_2, \ldots, \vec{u}_m\vec{v}_m$. It is easy to check that $s_{i,k}=1$. Hence, $S_A+S_B\geq \sum_{j\in [t]} \sum_{k\in [r]} 1= rt$.
\end{proof}

By Claim~\ref{cl:r3theta-2}, we have $|V(F^{\ast})\cap B|= p-(S_B-\sum_{i\in I_B'}s_i) +\sum_{i\in I_B'}d_i-\gamma= p-(S_B-\sum_{i\in I_B'}s_i) +\sum_{i\in I_B'}d_i-S_A= p-(S_A+S_B)+\sum_{i\in I_B'}(s_i+d_i)\leq p-rt+rt= p$. This completes the proof of Theorem~\ref{th:r3theta}.
\end{proof}

We next present our proof of Theorem~\ref{th:r=2theta}, which is similar to the proof of Theorem~\ref{th:r3theta}. The main difference is that we use the pruned second color energy graph $\widetilde{G}^2$ instead of $\widetilde{G}^r$, and we shall utilize property (ii) in Definition~\ref{def:prune energy_graph} of the pruned color energy graph.

\begin{proof}[Proof of Theorem~\ref{th:r=2theta}]
Let $p=(s-1)t+2$. Let $G=G(A,B)$ be an edge-colored $K_{n,n}$ such that every $K_{p,p}$ receives at least $p^2-st+1$ colors (i.e., every $K_{p,p}$ receives at most $st-1$ color repetitions). Our goal is to prove $|C(G)|=\Omega\left(n^{2-2/s}\right)$. So we may assume $|C(G)|=o\left(n^{2-2/s}\right)$; otherwise we are done. Moreover, we can deduce that $G$ contains no monochromatic copy of $K_{1,p-2}$ by analogous arguments as in the proof of Claim~\ref{cl:r3theta-1}. Therefore, the pruned second color energy graph $\widetilde{G}^2$ of $G$ exists. It suffices to prove that $\widetilde{G}^2$ contains no $\mit\Theta$$(s, 4s^2t)$.

Suppose for a contradiction that $\widetilde{G}^2$ contains a copy of $\mit\Theta$$(s, 4s^2t)$. Similarly as in the proof of Theorem~\ref{th:r3theta}, we get $H$, $\widetilde{R}_A=\{\vec{a}^{t+st+1}_{s-1}, \ldots, \vec{a}^{t+2st}_{s-1}\}$, $\widetilde{R}_B=\{\vec{b}^{t+1}_{1}, \ldots, \vec{b}^{t+st}_{1}\}$ and $\widetilde{G}$. Let $F\colonequals H\cup \pi(\widetilde{G})$ and $I_B'\colonequals \{s, 2s, \ldots, ts\}\subseteq I_B$. We still have $|V(F)\cap A|=~p-S_A-D_A$, $|V(F)\cap B|=~p-S_B-D_B+\sum_{i\in I_B'}s_i +\sum_{i\in I_B'}d_i$, and $F$ has at least $2st-D_A-D_B-\sum_{i\in [m]}\mathds{1}_{(d_i=0)}$ color repetitions. Moreover, we also have $S_A+S_B\geq 2t$ by the same arguments as in the proof of Claim~\ref{cl:r3theta-2}. Let $\gamma\colonequals  \min\{\sum_{i\in I_B'}d_i, S_A\}$. We divide the rest of the proof into two cases.

\medskip\noindent
{\bf Case 1.} $D_B-\gamma$ is even.
\vspace{0.05cm}

Let $D_1\colonequals  D_A+\gamma$ and $D_2\colonequals  D_B-\gamma$. By Lemma~\ref{le:reservoir}, there exists a graph $F^{\ast}$ such that $|V(F^{\ast})\cap A|=|V(F)\cap A|+D_1\leq p$, $|V(F^{\ast})\cap B|= |V(F)\cap B|+D_2\leq p$, and $F^{\ast}$ has at least $2st-D_A-D_B-\sum_{i\in [m]}\mathds{1}_{(d_i=0)}+\lfloor D_1/2\rfloor+\lfloor D_2/2\rfloor$ color repetitions. Since $D_2=D_B-\gamma$ is even, we have
\begin{align*}
~ &~2st-D_A-D_B-\sum_{i\in [m]}\mathds{1}_{(d_i=0)}+\lfloor D_1/2\rfloor+\lfloor D_2/2\rfloor \\
= &~2st-D_A-D_B-\sum_{i\in [m]}\mathds{1}_{(d_i=0)}+\lfloor D_1/2+D_2/2\rfloor \\
= &~2st-D_A-D_B-\sum_{i\in [m]}\mathds{1}_{(d_i=0)}+\lfloor (D_A+D_B)/2\rfloor \\
= &~2st-\sum_{i\in [m]}\mathds{1}_{(d_i=0)}-\lceil(D_A+D_B)/2\rceil \\
\geq &~2st-\sum_{i\in [m]}\mathds{1}_{(d_i=0)}-\sum_{i\in [m]}\mathds{1}_{(d_i\neq 0)}=~2st-m =~st.
\end{align*}
Thus $F^{\ast}$ is a subgraph of $K_{p,p}$ receiving at least $st$ color repetitions, a contradiction.

\medskip\noindent
{\bf Case 2.} $D_B-\gamma$ is odd.
\vspace{0.05cm}

Let $D_1\colonequals  D_A+\gamma-1$ and $D_2\colonequals  D_B-\gamma-1$. By Lemma~\ref{le:reservoir}, there exists a graph $F^{\ast}$ such that $|V(F^{\ast})\cap A|=|V(F)\cap A|+D_1\leq p-1$, $|V(F^{\ast})\cap B|= |V(F)\cap B|+D_2\leq p-1$, and $F^{\ast}$ has at least
\begin{align*}
~ &~2st-D_A-D_B-\sum_{i\in [m]}\mathds{1}_{(d_i=0)}+\lfloor D_1/2\rfloor+\lfloor D_2/2\rfloor \\
= &~2st-D_A-D_B-\sum_{i\in [m]}\mathds{1}_{(d_i=0)}+\lfloor D_1/2+D_2/2\rfloor \\
= &~2st-D_A-D_B-\sum_{i\in [m]}\mathds{1}_{(d_i=0)}+\lfloor (D_A+D_B-2)/2\rfloor \\
= &~2st-\sum_{i\in [m]}\mathds{1}_{(d_i=0)}-\lceil(D_A+D_B)/2\rceil-1 \\
\geq &~2st-\sum_{i\in [m]}\mathds{1}_{(d_i=0)}-\sum_{i\in [m]}\mathds{1}_{(d_i\neq 0)}-1=~2st-m-1=~st-1
\end{align*}
color repetitions. Let $e$ be an arbitrary edge of $F^{\ast}$. By Definition~\ref{def:prune energy_graph} (ii), there are at least $\log n$ edges colored by color $c(e)$ in $G$. We choose an arbitrary edge $f\in E(G)\setminus E(F^{\ast})$ with $c(f)=c(e)$. Let $F^{\ast\ast}$ be the graph obtained by adding $f$ to $F^{\ast}$. Then $|V(F^{\ast\ast})\cap A|\leq p$, $|V(F^{\ast\ast})\cap B|\leq p$, and $F^{\ast\ast}$ has at least $st$ color repetitions, a contradiction. This completes the proof of Theorem~\ref{th:r=2theta}.
\end{proof}


\section{Thresholds for linear and quadratic functions}\label{sec:unbalanced}

In this section, we study the thresholds for linear and quadratic functions $r(K_{n,n}, K_{s,t}, q)$. Firstly, we present our proof of Theorem~\ref{th:UnbalanceLin}, which concerns the linear threshold.

\begin{proof}[Proof of Theorem~\ref{th:UnbalanceLin}]
(i) We prove the upper bound by construction. Let $\ell=\lfloor(t-1)/(q-1)\rfloor$ and $k=\lceil n/\ell\rceil$. Let $G=G(A, B)$ be a copy of $K_{n,n}$. We partition $A$ (resp., $B$) into $k$ parts $A_1, A_2, \ldots, A_k$ (resp., $B_1, B_2, \ldots, B_k$) such that $|A_i|=|B_i|=\ell$ for all $i\in [k-1]$, and $|A_k|=|B_k|=n-(k-1)\ell$. We color the edges of $G$ using colors $1, 2, \ldots, k$ such that $c(A_i, B_j)\equiv i+j-1 \pmod{k}$ for every $1\leq i\leq j\leq k$. In the resulting edge-colored $K_{n,n}$, every $K_{1,t}$ is colored by at least $\lceil t/\ell\rceil\geq q$ colors. Thus $r(K_{n,n}, K_{1,t}, q)\leq k$.

For the lower bound, let $G=G(A, B)$ be an $r$-edge-colored $K_{n,n}$ such that every $K_{1,t}$ receives at least $q$ colors. We now show that $r\geq \lceil n(q-1)/(t-1)\rceil$. We fix a vertex $a\in A$ arbitrarily, and define $x_i\colonequals |\{b\in B\colon\, c(ab)=i\}|$ for each color $i\in [r]$. Then $\sum_{i=1}^{r}x_i=n$. Since every $K_{1,t}$ receives at least $q$ colors, for any $I\subseteq [r]$ with $|I|=q-1$ we have $\sum_{i\in I}x_i\leq t-1$. On the one hand, we have $\sum_{I\subseteq [r], |I|=q-1}\sum_{i\in I}x_i\leq {r \choose q-1}(t-1)$. On the other hand, we have $\sum_{I\subseteq [r], |I|=q-1}\sum_{i\in I}x_i\geq {r-1 \choose q-2}\sum_{i=1}^{r}x_i= {r-1 \choose q-2}n$. By straightforward calculations and since $r$ is an integer, we have $r\geq \lceil n(q-1)/(t-1)\rceil$.

(ii) For the upper bound, consider the following construction. Let $G=G(A, B)$ be a copy of $K_{n,n}$. We partition $A$ (resp., $B$) into $n-t+q$ parts $A_1, A_2, \ldots, A_{n-t+q}$ (resp., $B_1, B_2, \ldots, B_{n-t+q}$) such that $|A_{i_1}|=|B_{i_1}|=2$ for all $1\leq i_1\leq t-q$, and $|A_{i_2}|=|B_{i_2}|=1$ for all $t-q+1\leq i_2\leq n-t+q$. We color the edges of $G$ using colors $1, 2, \ldots, n-t+q$ such that $c(A_i, B_j)\equiv i+j-1 \pmod{n-t+q}$ for every $1\leq i\leq j\leq n-t+q$. Since $(t+2)/2\leq q\leq t$, every $K_{1,t}$ is colored by at least $q$ colors in the resulting edge-colored $K_{n,n}$. Thus $r(K_{n,n}, K_{1,t}, q)\leq n-t+q$.

For the lower bound, let $G=G(A, B)$ be an edge-colored $K_{n,n}$ such that every $K_{1,t}$ receives at least $q$ colors. Suppose $|C(G)|\leq n-t+q-1= n-(t-q+1)$. Let $a$ be an arbitrarily fixed vertex in $A$. Then there exists a subset $B'\subseteq B$ with $|B'|=t-q+1$ such that for any $b\in B'$, there exists $b^{\ast}\in B\setminus B'$ with $c(ab)=c(ab^{\ast})$. Note that $q\geq (t+2)/2$ implies $2(t-q+1)\leq t$. Thus there exists a $K_{1,t}$ which receives at most $t-|B'|=q-1$ colors, a contradiction. Thus $r(K_{n,n}, K_{1,t}, q)\geq n-t+q$.

(iii) The upper bounds follow from (\ref{eq:UnbalanGener}). For the lower bound, let $G$ be an edge-colored $K_{n,n}$ such that every $K_{s,t}$ receives at least $st-s-t+3$ colors. Let $T$ be the tree consisting of the disjoint union of stars $K_{1,s-1}$ and $K_{1,t-1}$ with an extra edge joining their centers. Then $T$ is an $(s+t-1)$-edge subgraph of $K_{s,t}$. Note that $G$ contains no monochromatic copy of $T$; otherwise $G$ contains a copy of $K_{s,t}$ receiving at most $st-s-t+2$ colors. Thus $|C(G)|\geq n^2/\mbox{ex}(n,n,T)$. A result of Sidorenko \cite{Sid89} implies that ex$(n, n, T)\leq $ ex$(2n, T)\leq 2n(|V(T)|-2)/2=(s+t-2)n$. Hence, $r(K_{n,n}, K_{s,t}, st-s-t+3)\geq n/(s+t-2)$.
\end{proof}

\begin{remark}\label{re:K1t} {\normalfont Theorem~\ref{th:UnbalanceLin} (i) implies that $r(K_{n,n}, K_{1,t}, q)=\lceil n(q-1)/(t-1)\rceil$ when $2\leq q\leq \frac{t+1}{2}$ and $(q-1) \mid (t-1)$. In particular, the exact values hold in the case $q=2$ and the case $q=3$ with odd $t$. In the case $(q-1) \nmid (t-1)$, we can prove the following slightly better lower and upper bounds:
$$r\leq r(K_{n,n}, K_{1,t}, q) \leq \lc\frac{n-t+1}{\lf (t-1)/(q-1)\rf }\rc+q-1,$$
where $r$ is the smallest positive integer satisfying
$$n{r-1 \choose q-2}\leq {r-q+2 \choose q-1}(q-1)\lf\frac{t-1}{q-1}\rf+ \left({r \choose q-1}-{r-q+2 \choose q-1}\right)(t-1).$$
Before providing our proof of these lower and upper bounds, we remark that these improved bounds imply that $r(K_{n,n}, K_{1,t}, q)=\lceil 2(n-1)/(t-2)\rceil$ in the case $q=3$ with even $t$. We first prove the lower bound. Let $G=G(A, B)$ be an $r$-edge-colored $K_{n,n}$ such that every $K_{1,t}$ receives at least $q$ colors. Similarly as in the proof of Theorem~\ref{th:UnbalanceLin} (i), we can define $x_i\colonequals |\{b\in B\colon\, c(ab)=i\}|$ for arbitrarily fixed vertex $a\in A$ and each $i\in [r]$. Without loss of generality, we may assume that $x_1\geq x_2\geq \cdots \geq x_r$. Since $(q-1) \nmid (t-1)$, we have $x_{q-1}\leq \lfloor(t-1)/(q-1)\rfloor$; otherwise there is a $(q-1)$-colored $K_{1,t}$. Thus for any $I\subseteq \{q-1, \ldots, r\}$ with $|I|=q-1$, we have $\sum_{i\in I}x_i\leq (q-1)\lfloor(t-1)/(q-1)\rfloor$. By double counting $\sum_{I\subseteq [r], |I|=q-1}\sum_{i\in I}x_i$, we can prove the lower bound. We next prove the upper bound by modifying the construction in the proof of Theorem~\ref{th:UnbalanceLin} (i). Let $\ell=\lfloor(t-1)/(q-1)\rfloor$, $\bar{\ell}=\lceil(t-1)/(q-1)\rceil$ and $k=\lceil (n-t+1)/\ell\rceil+q-1$. Let $t-1\equiv m \pmod{q-1}$, where $0\leq m\leq q-2$. We partition $A$ (resp., $B$) into $k$ parts $A_1, A_2, \ldots, A_k$ (resp., $B_1, B_2, \ldots, B_k$) such that $|A_{i_1}|=|B_{i_1}|=\bar{\ell}$ for all $1\leq i_1\leq m$, $|A_{i_2}|=|B_{i_2}|=\ell$ for all $m+1\leq i_2\leq k-1$, and $|A_k|=|B_k|=n-(k-1)\ell-m$. Note that $|A_k|=|B_k|\leq \ell$. We color the edges using colors $1, 2, \ldots, k$ such that $c(A_i, B_j)\equiv i+j-1 \pmod{k}$ for every $1\leq i\leq j\leq k$. In the resulting edge-colored $K_{n,n}$, every $K_{1,t}$ is colored by at least $q$ colors. This completes the proof of the upper bound.
}
\end{remark}

Next, we prove Theorem~\ref{th:UnbalanceQu3+}, which provides a lower bound on $r(K_{n,n}, K_{s,t}, q)$ for $q=st-\lfloor(s+t)/2\rfloor+1$.

\begin{proof}[Proof of Theorem~\ref{th:UnbalanceQu3+}] Let $G=G(A,B)$ be an edge-colored $K_{n,n}$ such that every $K_{s,t}$ receives at least $st-\lfloor(s+t)/2\rfloor+1$ colors (i.e., at most $\lfloor(s+t)/2\rfloor-1$ color repetitions). Let $A=\{a_1, a_2, \ldots, a_n\}$ and $B=\{b_1, b_2, \ldots, b_n\}$.

Since $s+2\leq t$, we have $\lfloor(s+t)/2\rfloor+1\leq \lfloor(t-2+t)/2\rfloor+1=t$. In order to avoid a $K_{s,t}$ with at least $\lfloor(s+t)/2\rfloor$ color repetitions, there is no monochromatic star with at least $\lfloor(s+t)/2\rfloor+1$ edges.

For each color $i\in C(G)$, let $G_i$ be the subgraph of $G$ induced by color $i$. Then the maximum degree $\Delta(G_i)$ of $G_i$ satisfies $\Delta(G_i)\leq \lfloor(s+t)/2\rfloor$ for all $i\in C(G)$. Let $m=\lceil n^{2/3}\rceil$. First suppose that $G_i$ contains no matching of size $m$ for all $i\in C(G)$. Then for each $i\in C(G)$, $G_i$ has a covering\footnote{A \emph{covering} of a graph is a set of vertices which together meet all edges of the graph. K\"{o}nig's Theorem states that in any bipartite graph, the size of a minimum covering is equal to the size of a maximum matching.} with at most $m-1$ vertices, and thus $|E(G_i)|\leq (m-1)\Delta(G_i)\leq (m-1)\lfloor(s+t)/2\rfloor=O(n^{2/3})$. Then $|C(G)|=\Omega(n^2/n^{2/3})=\Omega(n^{4/3})$, and we are done. Thus we may assume that there exists a color $i$ such that $G_i$ contains a matching $M$ of size $m$, say $M=\{a_1b_1, a_2b_2, \ldots,$ $a_mb_m\}$. Let $A'=\{a_1, a_2, \ldots, a_m\}$, $B'=\{b_1, b_2, \ldots, b_m\}$ and $p\colonequals \lfloor(s+t)/2\rfloor-s+1$. We will divide the rest of the proof into two cases.

\medskip\noindent
{\bf Case 1.} $s+2\leq t\leq 3s-3$ and $(s,t)\neq (3,5)$.
\vspace{0.05cm}

In this case, we have $p+1\leq \lfloor(s+3s-3)/2\rfloor-s+2=s$. We claim that $G[A'\cup B']$ contains at most $p$ edges of color $i'$ for each $i'\in C(G)\setminus \{i\}$. Otherwise, without loss of generality, we may assume that $c(a'_1b_1')=c(a'_2b_2')=\cdots =c(a'_{p+1}b_{p+1}')=i'$ for some $i'\in C(G)\setminus \{i\}$, where $a_1', a_2', \ldots, a_{p+1}'\in \{a_1, a_2, \ldots, a_{p+1}\}$ and $b_1', b_2', \ldots, b_{p+1}'\in B'$ (note that $a_1', a_2', \ldots, a_{p+1}'$ (resp., $b_1', b_2', \ldots, b_{p+1}'$) are not necessarily pairwise distinct).

If $t\geq s+3$, then $|\{b_1, b_2, \ldots, b_s\}\cup \{b_1', b_2', \ldots, b_{p+1}'\}|\leq s+\lfloor(s+t)/2\rfloor-s+2\leq \lfloor(t-3+t)/2\rfloor+2= t$. Thus $G[\{a_1, a_2, \ldots, a_s,$ $b_1, b_2, \ldots, b_s\}\cup \{b_1', b_2', \ldots, b_{p+1}'\}]$ is a subgraph of $K_{s,t}$ with at least $s-1+p=\lfloor(s+t)/2\rfloor$ color repetitions, a contradiction. If $t=s+2$, then $s\geq 4$ since $(s,t)\neq (3,5)$. Now we have $p+1= \lfloor(s+s+2)/2\rfloor-s+2=3\leq s-1$ and $|\{b_1, b_2, \ldots, b_{s-1}\}\cup \{b_1', b_2', b_{3}'\}|\leq s-1+3=t$. Let $e=ab_1'$ be the edge in $M$ incident with $b_1'$, where $a\in A'$. Note that $\{b_1', b_2', b_{3}'\}\cap \{b_1, b_2, \ldots, b_{s-1}\}= \emptyset$; otherwise $|\{b_1, b_2, \ldots, b_{s-1}\}\cup \{b_1', b_2', b_{3}'\}|\leq t-1$, which implies that $G[\{a_1, a_2, \ldots, a_{s-1}, a_s, b_1, b_2, \ldots, b_{s-1}, b_s\}\cup \{b_1', b_2', b_{3}'\}]$ is a subgraph of $K_{s,t}$ with at least $s+1$ color repetitions, a contradiction. Thus $a\notin \{a_1, a_2, \ldots, a_{s-1}\}$. Now $G[\{a_1, a_2, \ldots, a_{s-1}, a, b_1, b_2, \ldots, b_{s-1}\}\cup \{b_1', b_2', b_{3}'\}]$ is a subgraph of $K_{s,t}$ with at least $s+1$ color repetitions, a contradiction.

Hence, $G[A'\cup B']$ contains at most $p$ edges of color $i'$ for any $i'\in C(G)\setminus \{i\}$. Now $|C(G)|\geq |C(G[A'\cup B'])|\geq (m^2-m\lfloor(s+t)/2\rfloor)/p+1=\Omega(n^{4/3})$.

\medskip\noindent
{\bf Case 2.} $t\geq 3s-2$ and $(s,t)\neq (3,7)$.
\vspace{0.05cm}

For each $1\leq j\leq m$, let $C_j\colonequals \{i\in C(G)\colon\, c(a_jb_{\ell})=i \mbox{ for some } m+1\leq \ell \leq n\}$. Recall that there is no monochromatic star with at least $\lfloor(s+t)/2\rfloor+1$ edges. Thus $|C_j|\geq (n-m)/\lfloor(s+t)/2\rfloor$ for every $j\in [m]$. Suppose that $|C_{j_1}\cap C_{j_2}\cap C_{j_3}|\geq \lceil p/2\rceil$ for some $1\leq j_1<j_2<j_3\leq m$. If $p$ is even, then $3\lceil p/2\rceil+s=3p/2+s\leq t$ by Lemma~\ref{le:UnbalanceQu3+}. Then $G$ contains a subgraph of $K_{s,t}$ with at least $2\lceil p/2\rceil+s-1=p+s-1=\lfloor(s+t)/2\rfloor$ color repetitions, a contradiction. If $p$ is odd, then $3\lceil p/2\rceil+s-1=3(p+1)/2+s-1\leq t$ by Lemma~\ref{le:UnbalanceQu3+}. Then $G$ contains a subgraph of $K_{s,t}$ with at least $2\lceil p/2\rceil+s-2=p+s-1=\lfloor(s+t)/2\rfloor$ color repetitions, a contradiction. Thus $|C_{j_1}\cap C_{j_2}\cap C_{j_3}|<\lceil p/2\rceil$ for every $1\leq j_1<j_2<j_3\leq m$. Since $|C(G)|\geq |C_1\cup C_2\cup \cdots \cup C_m|$ and by Lemma~\ref{le:gCorradi}, we have
$$|C(G)|\geq \left(\frac{((n-m)/\lfloor(s+t)/2\rfloor)^3m^2}{(3m-2)(n-m)/\lfloor(s+t)/2\rfloor+(m-1)(m-2)\lceil p/2\rceil}\right)^{1/2}=\Omega(n^{4/3}).$$
\end{proof}

Thirdly, we present our proof of Theorem~\ref{th:UnbalanceQu1}, which concerns the quadratic threshold.

\begin{proof}[Proof of Theorem~\ref{th:UnbalanceQu1}] The upper bound $O(n^2)$ is trivial in all four cases, so we only prove the lower bound $\Omega(n^2)$. We first prove $r(K_{n,n}, K_{s,t},$ $st-s+2)=\Omega(n^2)$. Let $G$ be an edge-colored $K_{n,n}$ such that every $K_{s,t}$ receives at least $st-s+2$ colors. Note that $G$ contains neither a monochromatic star $K_{1,s}$ nor a monochromatic matching $sK_2$; otherwise there is a $K_{s,t}$ receiving at most $st-s+1$ colors. For each color $i\in C(G)$, let $G_i$ be the subgraph of $G$ induced by color $i$. Since $G$ contains no monochromatic copy of $K_{1,s}$, we have $\Delta(G_i)\leq s-1$. Since $G$ contains no monochromatic copy of $sK_2$, $G_i$ has a covering of size at most $s-1$. Thus $|C(G)|\geq n^2/|E(G_i)|\geq n^2/(s-1)^2$.

Next, we give a unified proof of the lower bounds in (ii), (iii) and (iv). Set
$$r\colonequals
\left\{
   \begin{aligned}
    &\lfloor t/2\rfloor-1, & & \mbox{if $2\leq s\leq 3$},\\
    &\lfloor(s+t)/2\rfloor-3, & & \mbox{if $s\geq 4$ and at least one of $s$ and $t$ is even},\\
    &\lfloor(s+t)/2\rfloor-4, & & \mbox{if $s\geq 5$ and both $s$ and $t$ are odd}.
   \end{aligned}
   \right.$$
Let $G=G(A,B)$ be an edge-colored $K_{n,n}$ such that every $K_{s,t}$ receives at least $st-r$ colors (i.e., every $K_{s,t}$ receives at most $r$ color repetitions). Suppose for a contradiction that $|C(G)|=o(n^2)$. Let $A=\{a_1, a_2, \ldots, a_n\}$ and $B=\{b_1, b_2, \ldots, b_n\}$. For each $j\in [n]$, let $X_j\colonequals \{i\in C(G)\colon\, c(a_jb)=i \mbox{ for some } b\in B\}$ and $Y_j\colonequals \{i\in C(G)\colon\, c(b_ja)=i \mbox{ for some } a\in A\}$. Note that $C(G)=X_1\cup X_2\cup \cdots \cup X_n= Y_1\cup Y_2\cup \cdots \cup Y_n$.

\begin{claim}\label{cl:UnbalanceQu1-1} For every $j\in [n]$, we have $|X_j|\geq n/(r+1)$ and $|Y_j|\geq n/(r+1)$.
\end{claim}

\begin{proof} By symmetry, we only prove the statement for $X_j$. Suppose $|X_j|< n/(r+1)$ for some $j\in [n]$. Then there exist $r+2\leq t$ edges incident with $a_j$ of the same color. Thus $G$ contains a $K_{s,t}$ receiving at least $r+1$ color repetitions, a contradiction.
\end{proof}

\begin{claim}\label{cl:UnbalanceQu1-2} For every $1\leq j_1<j_2\leq n$, we have $|X_{j_1}\cap X_{j_2}|\leq \lfloor t/2\rfloor-1$.
\end{claim}

\begin{proof} We first consider the case $2\leq s\leq 3$. Suppose that $|X_{j_1}\cap X_{j_2}|\geq \lfloor t/2\rfloor=r+1$ for some $1\leq j_1<j_2\leq n$. Then $G$ contains a subgraph of $K_{s,t}$ with at least $r+1$ color repetitions, a contradiction.

Next, we consider the case $s\geq 4$. In this case, we shall prove that $|X_{j_1}\cap X_{j_2}|< \lfloor (t-2)/2\rfloor$ for every $1\leq j_1<j_2\leq n$. Suppose that $|X_{j_1}\cap X_{j_2}|\geq \lfloor (t-2)/2\rfloor$ for some $1\leq j_1<j_2\leq n$. This implies that there exists a subset $B'\subseteq B$ of size at most $2\lfloor (t-2)/2\rfloor$ (say $B'\subseteq \{b_{1}, b_{2}, \ldots, b_{2\lfloor (t-2)/2\rfloor}\}$) such that $G[\{a_{j_1}, a_{j_2}\}\cup B']$ has at least $\lfloor (t-2)/2\rfloor$ color repetitions. We claim that for any $2\lfloor (t-2)/2\rfloor+1\leq \ell_1< \ell_2\leq n$, we have $|Y_{\ell_1}\cap Y_{\ell_2}|< \lfloor(s-2)/2\rfloor$. Otherwise if $|Y_{\ell_1}\cap Y_{\ell_2}|\geq \lfloor(s-2)/2\rfloor$ for some $2\lfloor (t-2)/2\rfloor+1\leq \ell_1< \ell_2\leq n$, then there exists a subset $A'\subseteq A$ of size at most $2\lfloor (s-2)/2\rfloor$ such that $G[A'\cup \{b_{\ell_1}, b_{\ell_2}\}]$ has at least $\lfloor (s-2)/2\rfloor$ color repetitions. Since $2\lfloor (t-2)/2\rfloor+2\leq t$, $2\lfloor (s-2)/2\rfloor +2\leq s$ and $\lfloor (t-2)/2\rfloor+\lfloor (s-2)/2\rfloor=r+1$, we have that $G[\{a_{j_1}, a_{j_2}, b_{\ell_1}, b_{\ell_2}\}\cup A'\cup B']$ is a subgraph of $K_{s,t}$ with at least $r+1$ color repetitions, a contradiction. Thus $|Y_{\ell_1}\cap Y_{\ell_2}|< \lfloor(s-2)/2\rfloor$ for any $2\lfloor (t-2)/2\rfloor+1\leq \ell_1< \ell_2\leq n$. By Claim~\ref{cl:UnbalanceQu1-1} and Lemma~\ref{le:Corradi}, we have
$$|C(G)|\geq \left|Y_{2\lfloor (t-2)/2\rfloor+1}\cup \cdots \cup Y_n\right|> \frac{(n/(r+1))^2(n-2\lfloor (t-2)/2\rfloor)}{n/(r+1)+(n-2\lfloor (t-2)/2\rfloor-1)\lfloor(s-2)/2\rfloor}=\Omega(n^2).$$
This contradiction completes the proof of Claim~\ref{cl:UnbalanceQu1-2}.
\end{proof}
By Lemma~\ref{le:Corradi}, Claims~\ref{cl:UnbalanceQu1-1} and \ref{cl:UnbalanceQu1-2}, we have
$$|C(G)|\geq \left|X_{1}\cup \cdots \cup X_n\right|\geq \frac{(n/(r+1))^2n}{n/(r+1)+(n-1)(\lfloor t/2\rfloor-1)}=\Omega(n^2).$$
This contradiction completes the proof of Theorem~\ref{th:UnbalanceQu1}.
\end{proof}

Finally, we prove Theorem~\ref{th:UnbalanceQu2}, which provides a lower bound on $r(K_{n,n}, K_{s,t}, q)$ for $q=st-\lfloor(s+t)/2\rfloor+2$.

\begin{proof}[Proof of Theorem~\ref{th:UnbalanceQu2}] Let $p\colonequals \lfloor(s+t)/2\rfloor-1$. Let $G=G(A,B)$ be an edge-colored $K_{n,n}$ such that every $K_{s,t}$ receives at least $st-p+1$ colors. Equivalently, every $K_{s,t}$ receives at most $p-1$ color repetitions. Let $A=\{a_1, a_2, \ldots, a_n\}$ and $B=\{b_1, b_2, \ldots, b_n\}$.

For each color $i\in C(G)$, let $G_i$ be the subgraph of $G$ induced by color $i$. Since every $K_{s,t}$ receives at most $p-1$ color repetitions, there is no monochromatic star $K_{1,p+1}$. Hence, for every $i\in C(G)$, we have $\Delta(G_i)\leq p$. Let $m=\lceil n^{1/2}\rceil$. If there is no monochromatic matching of size $m$, then every $G_i$ has a covering with at most $n^{1/2}$ vertices. This implies that $|E(G_i)|\leq n^{1/2}\Delta(G_i)\leq pn^{1/2}$ for all $i\in C(G)$. Then $|C(G)|\geq n^2/(pn^{1/2})=\Omega(n^{3/2})$, and we are done. Thus there exists a monochromatic matching $M$ of size $m$, say $M=\{a_1b_1, a_2b_2, \ldots, a_{m}b_{m}\}$.

For each $1\leq j\leq m$, let $X'_j\colonequals \{i\in C(G)\colon\, c(a_jb_{\ell})=i \mbox{ for some } m+1\leq \ell \leq n\}$. Recall that there is no monochromatic star $K_{1,p+1}$. Thus $|X'_j|\geq (n-m)/p$ for all $1\leq j\leq m$. Note that $2(p-s+1)+s=2(\lfloor(s+t)/2\rfloor-s)+s\leq t$. Suppose that for some $1\leq j_1<j_2\leq m$, we have $|X'_{j_1}\cap X'_{j_2}|\geq p-s+1$, say $j_1=1$ and $j_2=2$. Then there exist $B_1, B_2\subseteq \{b_{m+1}, \ldots, b_n\}$ with $|B_1|=|B_2|=p-s+1$ and $C(a_1, B_1)=C(a_2, B_2)$. Then $\{a_1, a_2, \ldots, a_s\}\cup \{b_1, b_2, \ldots, b_s\}\cup B_1\cup B_2$ forms a subgraph of $K_{s,t}$ with at least $p-s+1+s-1=p$ color repetitions, a contradiction. Thus $|X'_{j_1}\cap X'_{j_2}|\leq p-s$ for every $1\leq j_1<j_2\leq m$. By Lemma~\ref{le:Corradi}, we have $$|C(G)|\geq |X'_1\cup X'_2\cup \cdots \cup X'_{m}|\geq \frac{((n-m)/p)^2m}{(n-m)/p + (m-1)(p-s)}=\Omega(n^{3/2}).$$
\end{proof}


\section{Proofs using and without using the Color Energy Method}\label{sec:unbalanced2}

In this section, we show that certain results can be proved using the Color Energy Method, but also without using the Color Energy Method. We demonstrate this for Theorem~\ref{th:Ksst}. Our first proof uses the Color Energy Method and is based on an idea that we got from a result of \cite{PoSh}.

\begin{proof}[First proof of Theorem~\ref{th:Ksst}] Let $G=G(A,B)$ be an edge-colored $K_{n,n}$ such that every $K_{s,st}$ receives at least $s^2t-t(s-1)+1$ colors (i.e., at most $t(s-1)-1$ color repetitions). Note that $t(s-1)+1\leq st$. In order to avoid a subgraph of $K_{s,st}$ with at least $t(s-1)$ color repetitions, we have the following two obvious facts.

\begin{fact}\label{fa:Ksst2nd1} There is no monochromatic $K_{1, t(s-1)+1}$ in $G$.
\end{fact}

\begin{fact}\label{fa:Ksst2nd2} There cannot be $s$ vertices in one partite set of $G$ that are incident with the same $t$ colors.
\end{fact}

For each color $i\in C(G)$, let $m_i$ be the number of edges of color $i$ in $G$ and let $A_i\colonequals \{u\in A\colon\, c(uv)=i \mbox{ for some } v\in B\}$. By Fact~\ref{fa:Ksst2nd1}, we have $m_i\leq t(s-1)n$ and $|A_i|\geq m_i/(t(s-1))$.

For each integer $0\leq j\leq \log_2(stn)$, let $C_{j}\colonequals \{i\in C(G)\colon\, m_i\geq 2^j\}$ and $k_j\colonequals |C_j|$. Since $|E(G)|=n^2$, it is obvious that $k_j\leq n^2/2^j$. For any $j$ with $2^j\geq t(2s^{t+1}n^{t-1})^{1/t}$, we now give a better upper bound on $k_j$. If $k_j<t$, then clearly $k_j<2t^{t+1}s^tn^t/2^{jt}$ since $j\leq \log_2(stn)$. We next consider the case $k_j\geq t$. For any $t$ distinct colors $c_1, c_2, \ldots, c_t\in C_j$, we have $|A_{c_1}\cap A_{c_2}\cap \cdots \cap A_{c_t}|<s\leq (2^j/(st))^t/(2n^{t-1})$ by Fact~\ref{fa:Ksst2nd2} and since $2^j\geq t(2s^{t+1}n^{t-1})^{1/t}$. Recall that $|A_{i}|\geq m_{i}/(t(s-1))\geq 2^j/(st)$ for every $i\in C_j$. The contrapositive of Lemma~\ref{le:Erd64} implies that $k_j<2t(n/(2^j/st))^t= 2t^{t+1}s^tn^t/2^{jt}$.

Finally, we prove an upper bound on the number of edges of the $r$th color energy graph $G^r$, where $2\leq r\leq t$. Let $\ell \colonequals  \lfloor\log_2 (t(2s^{t+1}n^{t-1})^{1/t})\rfloor$. Then
\begin{align*}
 |E(G^r)|= &~\sum^{\log_2(stn)}_{j=0} \sum_{i\in C(G),~2^j\leq m_i< 2^{j+1}} m^r_i < \sum^{\log_2(stn)}_{j=0} k_j2^{(j+1)r} = \sum^{\ell}_{j=0} k_j2^{(j+1)r}+ \sum^{\log_2(stn)}_{j=\ell+1} k_j2^{(j+1)r} \\
  \leq &~\sum^{\ell}_{j=0} \frac{n^2}{2^j}2^{(j+1)r}+ \sum^{\log_2(stn)}_{j=\ell+1} \frac{2t^{t+1}s^tn^t}{2^{jt}}2^{(j+1)r} \\
  \leq &~n^22^r\left(2^0+2^{r-1}+\cdots +2^{\ell (r-1)}\right) +\sum^{\log_2(stn)}_{j=\ell+1} 2^{r+1}t^{t+1}s^tn^t2^{j(r-t)} \\
  = &~O\left(n^22^{\ell (r-1)}\right)+O\left(n^t2^{\ell (r-t)}\log_2 n\right) \\
  = &~O\left(n^{2+(1-1/t)(r-1)}\right)+O\left(n^{t+(1-1/t)(r-t)}\log_2 n\right)= O\left(n^{r+1-(r-1)/t}\right).
\end{align*}
By Proposition~\ref{prop:|C(G)|}, we have $|C(G)|=\Omega\left((n^{2r}/|E(G^r)|)^{1/(r-1)}\right)= \Omega\left((n^{2r}/n^{r+1-(r-1)/t})^{1/(r-1)}\right)$ $= \Omega\left(n^{1+1/t}\right)$. This completes the proof of Theorem~\ref{th:Ksst}.
\end{proof}

We next give an alternative proof without using the Color Energy Method. Theorem~\ref{th:Ksst} is in fact a corollary of the following lemma. The idea of this lemma comes from a result of \cite{BEHK}.

\begin{lemma}\label{le:K|U||EH|} Let $n \ll m \ll n^2$, and let $H=H(U,V)$ be a bipartite graph with $|E(H)|>|V|\geq 2$. Suppose that every subgraph of $K_{n, m}$ with $\frac{n^2}{|E(H)|-|V|}$ edges contains a copy of $H$ with $U$ in the partite set of size $n$ and $V$ in the partite set of size $m$. Then $$r\big(K_{n,n}, K_{|U|,|E(H)|}, |U||E(H)|-(|E(H)|-|V|)+1\big)> m.$$
\end{lemma}

\begin{proof} Let $G=G(A,B)$ be an edge-colored $K_{n,n}$ such that every $K_{|U|,|E(H)|}$ receives at least $|U||E(H)|-(|E(H)|-|V|)+1$ colors. Suppose for a contradiction that $|C(G)|\leq m$. We claim that each vertex is incident with at most $|E(H)|-|V|$ edges of the same color in $G$. Otherwise $G$ contains a subgraph of $K_{|U|,|E(H)|}$ with at least $|E(H)|-|V|$ color repetitions, which is a contradiction.

We define the color incidence graph $\mathscr{C}(G)$ of $G$ as follows. The graph $\mathscr{C}(G)$ is a bipartite graph with partite sets $A$ and $C(G)$, and $u\in A$ and $i\in C(G)$ are adjacent in $\mathscr{C}(G)$ if and only if $u$ is incident with color $i$ in $G$. From the above arguments, we have $|E(\mathscr{C}(G))|\geq \sum_{u\in A}\frac{d_G(u)}{|E(H)|-|V|}=\frac{n^2}{|E(H)|-|V|}$. Then $\mathscr{C}(G)$ contains a copy of $H$ with $U\subseteq A$ and $V\subseteq C(G)$.

Therefore, for each vertex $u\in U\subseteq A$, there is a star $S_u$ in $G$ with $d_H(u)$ edges and centered at $u$, such that each edge of $S_u$ is colored by a color in $V$. Note that for all $u\in U$, these stars are pairwise edge-disjoint. Thus $G$ contains a subgraph of $K_{|U|,|E(H)|}$ with at least $|E(H)|-|V|$ color repetitions, a contradiction.
\end{proof}

\begin{proof}[Second proof of Theorem~\ref{th:Ksst}] Let $H=H(U,V)$ be a copy of $K_{s,t}$ with $|U|=s$ and $|V|=t$. By Theorem~\ref{th:Tu} (iv), we may choose a constant $\alpha$ such that $n^2/(st-t)>\mbox{z}(\alpha n^{1+1/t}, n; t, s)$. Then every subgraph of $K_{n, \alpha n^{1+1/t}}$ with $\frac{n^2}{|E(H)|-|V|}$ edges contains a copy of $H$ with $U$ in the partite set of size $n$ and $V$ in the partite set of size $\alpha n^{1+1/t}$. By Lemma~\ref{le:K|U||EH|}, we have $r(K_{n,n}, K_{s,st}, s^2t-t(s-1)+1)> \alpha n^{1+1/t}$.
\end{proof}

\begin{remark}\label{re:Ksst} {\normalfont In both of our two proofs of Theorem~\ref{th:Ksst}, we rely on the properties of $G$ described in Facts~\ref{fa:Ksst2nd1} and \ref{fa:Ksst2nd2}. The difference is how to derive a lower bound on $|C(G)|$. In the first proof, we derive a lower bound on $|C(G)|$ by proving an upper bound on the number of edges of $G^r$. Facts~\ref{fa:Ksst2nd1} and \ref{fa:Ksst2nd2} are used to provide an upper bound on $k_j$ (and thus on $|E(G^r)|$). However, in the second proof we use bipartite Tur\'{a}n-type results to show that the larger part (corresponding to $C(G)$) of the color incidence graph should have size greater than $\alpha n^{1+1/t}$. Facts~\ref{fa:Ksst2nd1} and \ref{fa:Ksst2nd2} are related to the number of edges of the color incidence graph $\mathscr{C}(G)$ as well as the structure of the forbidden subgraph $H$.
}
\end{remark}


\section{Results obtained without using the Color Energy Method}\label{sec:not-energy}

In this section, we present our proofs of Theorems~\ref{th:2p-2}, \ref{th:Kab*} and \ref{th:induction}, which all do not use the Color Energy Method. Firstly, we prove Theorem~\ref{th:2p-2} using the bipartite Tur\'{a}n number for an even cycle.

\begin{proof}[Proof of Theorem~\ref{th:2p-2}] Let $G$ be an edge-colored $K_{n,n}$ such that every $K_{p,p}$ receives at least $p^2-2p+2$ colors. It suffices to show that every color appears $O(n^{1+1/p})$ times in $G$. Suppose that there exists a color $c_0$ which is used on $\Omega(n^{1+1/p})$ edges. Let $G'$ be the spanning subgraph of $G$ whose edge set is the set of all edges of color $c_0$ in $G$. By Theorem~\ref{th:Tu} (iii), we have ex$(n,n, C_{2p})=O(n^{1+1/p})$. We may choose a sufficiently large constant in the $\Omega(\cdot)$-notation such that $|E(G')|\geq $ ex$(n,n, C_{2p})$. Thus $G$ contains a monochromatic copy of $C_{2p}$. Then $G$ contains a $K_{p,p}$ using at most $p^2-2p+1$ colors, a contradiction.
\end{proof}

Next, we prove the following refined version of Theorem~\ref{th:Kab*}.

\begin{theorem}\label{th:Kab*+} For integers $s, t, a$ and $b$ with $2\leq a\leq s$ and $2\leq b\leq t$, we have $r(K_{n,n}, K_{s,t}, st-ab+2)\geq (1-o(1))\left(\frac{n}{\max\{a,b\}-1}\right)^{1/\min\{a,b\}}$.
\end{theorem}

\begin{proof} Without loss of generality, we may assume that $a\leq b$. Let $G$ be an edge-colored $K_{n,n}$ such that every $K_{s,t}$ receives at least $st-ab+2$ colors. Then $G$ contains no monochromatic copy of $K_{a,b}$. Thus every color appears at most $\mbox{z}(n, n; a, b)$ times in $G$. By Theorem~\ref{th:Tu} (iv), we have $\mbox{z}(n, n; a, b)\leq (b-1)^{1/a}(n-a+1)n^{1-1/a}+(a-1)n$. Thus $|C(G)|\geq n^2/\mbox{z}(n, n; a, b)\geq (1-o(1))(n/(b-1))^{1/a}$.
\end{proof}

Following from a result of Axenovich, F\"{u}redi and Mubayi in \cite{AxFM}, the lower bound given in Theorem~\ref{th:Kab*+} is asymptotically sharp in the case $a=s=2$ and $b=t$. In fact, Chung and Graham \cite{ChGr75} conjectured that this lower bound is asymptotically sharp for $a=s$ and $b=t$.

Finally, we prove the following refined version of Theorem~\ref{th:induction}.

\begin{theorem}\label{th:induction+} For integers $s, t$ and $a$ with $2\leq a\leq s\leq t$ and $a(s+t-a-2)\geq 1$, we have $r\big(K_{n,n}, K_{s,t}, st-a(s+t-a-2)+1\big)\geq \big(\frac{n}{t-1}\big)^{1/a}$.
\end{theorem}

\begin{proof} Let $G=G(U,V)$ be an edge-colored $K_{n,n}$ such that $|C(G)|< (n/(t-1))^{1/a}$, where $U=\{u_1, u_2, \ldots, u_n\}$ and $V=\{v_1, v_2, \ldots, v_n\}$. It suffices to show that $G$ contains a copy of $K_{s,t}$ with at most $st-a(s+t-a-2)$ colors. Note that there exists a subset $V_1\subseteq V$ with $|C(u_1, V_1)|=1$ and $|V_1|\geq n/|C(G)|> n/(n/(t-1))^{1/a}=n^{1-1/a}(t-1)^{1/a}$. Then there exists a subset $V_2\subseteq V_1$ with $|C(u_2, V_2)|=1$ and $|V_2|\geq |V_1|/|C(G)|> n^{1-1/a}(t-1)^{1/a}/(n/(t-1))^{1/a}= n^{1-2/a}(t-1)^{2/a}$. Continuing with this process, we get subsets $V_a\subseteq V_{a-1}\subseteq \cdots \subseteq V_1\subseteq V$ with $|V_a|> n^{1-a/a}(t-1)^{a/a}=t-1$ and $|C(u_i, V_a)|=1$ for all $i\in [a]$. Since $|V_a|$ is an integer, we have $|V_a|\geq t$. We choose $V'\subseteq V_a$ with $|V'|=t$ arbitrarily, say $V'=\{v_1, v_2, \ldots, v_t\}$. If $a=s$, then $G[V'\cup \{u_1, u_2, \ldots, u_a\}]$ is a $K_{s,t}$ with at most $a$ colors. Since $st-a(s+t-a-2)=2a$ in this case, we are done. Next assume $a<s$.

Let $U_0=U\setminus \{u_1, u_2, \ldots, u_a\}$. Note that there exists a subset $U_1\subseteq U_0$ with $|C(v_1, U_1)|=1$ and $|U_1|\geq |U_0|/|C(G)|$. Then there exists a subset $U_2\subseteq U_1$ with $|C(v_2, U_2)|=1$ and $|U_2|\geq |U_1|/|C(G)|\geq |U_0|/|C(G)|^2$. Continuing with this process, we get subsets $U_a\subseteq U_{a-1}\subseteq \cdots \subseteq U_1\subseteq U_0$ with $|U_a|\geq |U_0|/|C(G)|^a> (n-a)(t-1)/n\geq s-a$ and $|C(v_i, U_a)|=1$ for all $i\in [a]$. Thus we can arbitrarily choose $U'\subseteq U_a$ with $|U'|=s-a$. Then $G[U'\cup V'\cup \{u_1, u_2, \ldots, u_a\}]$ is a $K_{s,t}$ with at most $a+a+(s-a)(t-a)=st-a(s+t-a-2)$ colors. This completes the proof.
\end{proof}

\section{Concluding remarks}\label{sec:concluding}

In this paper, we studied the behavior of the function $r(K_{n,n}, K_{s,t}, q)$, which is a generalization of the multicolor bipartite Ramsey number. In particular, we showed that $r(K_{n,n}, K_{1,t}, q)$ is linear in $n$ for all $2\leq q\leq t$, and that $q=st-s-t+3$ is the threshold for linear $r(K_{n,n}, K_{s,t}, q)$ when $t\geq s\geq 2$. Moreover, we showed that the threshold for quadratic $r(K_{n,n}, K_{s,t}, q)$ is between $st-\lfloor(s+t)/2\rfloor+2$ and $st-\lfloor(s+t)/2\rfloor+4$ when $t\geq s\geq 2$. We propose the following problem and conjecture related to the threshold for linear and quadratic growth of this function. We solved this problem and conjecture for several special cases, and leave the other cases as open problems.

\begin{problem}\label{prob:lin} For $2\leq q\leq \frac{t+1}{2}$, determine the exact value of $r(K_{n,n}, K_{1,t}, q)$.
\end{problem}

\begin{conjecture}\label{conj:quad} Let $t\geq s\geq 2$ be two integers. Then $r(K_{n,n}, K_{s,t}, st-\lfloor(s+t)/2\rfloor+2)=\Theta(n^2).$
\end{conjecture}

We also improved some known lower bounds given by Axenovich, F\"{u}redi and Mubayi \cite{AxFM}, and obtained some nontrivial lower bounds for new families of triples $(s,t,q)$. Some of our proofs rely on our extension of the Color Energy Method to bipartite graphs, which is introduced in Section~\ref{sec:pre} and enhanced in Section~\ref{sec:adv}. Next we give two explanations to illustrate the differences between studying $f(n, p, q)$ and $r(K_{n,n}, K_{s,t}, q)$ using the Color Energy Method. These differences indicate why it turns out to be more difficult to study the problems for bipartite graphs.

Firstly, when applying this method to study $f(n, p, q)$, one usually needs to show the existence of a copy of $K_{p}$ with at most $q-1$ colors (i.e., at least ${p \choose 2}-q+1$ color repetitions) in an edge-colored $K_{n}$. To this end, often one first shows the existence of a $K_{p'}$ with ${p \choose 2}-q+1-r$ color repetitions, and then extends this to a desired $K_p$ by a $K_{p-p'}$ with $r$ color repetitions disjoint from this $K_{p'}$. When applying the same approach to study $r(K_{n,n}, K_{s,t}, q)$, the target would be a copy of $K_{s,t}$ with at most $q-1$ colors (i.e., at least $st-q+1$ color repetitions) in an edge-colored $K_{n,n}$. Assuming one would take a similar approach of first showing the existence of a $K_{s',t'}$ with $st-q+1-r$ color repetitions. Then the next step would be to extend this $K_{s',t'}$ to obtain the target by adding a subgraph of order $s+t-s'-t'$ with $r$ color repetitions. However, the obvious difficulty in the bipartite case is to guarantee that the resulting graph is the desired $K_{s,t}$.

Secondly, let $G$ be an edge-colored $K_n$ such that every $K_{p}$ receives at most $r$ color repetitions, where $r\leq p-3$. Then $G$ has the
nice property that it contains no monochromatic $K_{1, p-1}$. Let $G'$ be an edge-colored $K_{n,n}$ such that every $K_{s,t}$ receives at most $r$ color repetitions, where $r\leq s+t-3$. Then $G'$, however, does not necessarily have the property that it contains no monochromatic $K_{1, s+t-1}$.

In Section~\ref{sec:unbalanced2}, we showed that Theorem~\ref{th:Ksst} can be proved using the Color Energy Method, but also without using the Color Energy Method. For other results, such as Theorems~\ref{th:C_p}, \ref{th:r=2theta}, \ref{th:r3theta} and \ref{th:Kt+}, we were not able to find a proof without using the Color Energy Method.

We close this paper by a remark on Theorem~\ref{th:C_p}, which states that $r(K_{n,n}, K_{p,p}, p^2-p+1)=\Omega\left(n^{2-2/\lfloor p/2\rfloor}\right)$ for $p\geq 2$. If one can prove an upper bound on $r(K_{n,n}, K_{p,p}, p^2-p+1)$ of the form $O\left(n^{2-2/\lfloor p/2\rfloor}\right)$, then this would yield a lower bound on the Tur\'{a}n number ex$(n, C_{2k})$ of the form $\Omega\left(n^{1+1/k}\right)$. Determining good lower bounds for ex$(n, C_{2k})$ is a long-standing open problem, and it was conjectured to be $\Omega\left(n^{1+1/k}\right)$ by Erd\H{o}s and Simonovits \cite{ErSi}.

\begin{problem}\label{prob:evencycle} For any integer $p\geq 2$, is $r(K_{n,n}, K_{p,p}, p^2-p+1)=\Theta\left(n^{2-2/\lfloor p/2\rfloor}\right)$?
\end{problem}

\appendix

\section*{Appendix}

\section{Proof of Lemma~\ref{le:UnbalanceQu3+}}\label{ap:1}

For convenience, we restate Lemma~\ref{le:UnbalanceQu3+}.

\medskip\noindent{\bf Lemma~\ref{le:UnbalanceQu3+}.} {\it Let $s, t$ be integers with $s\geq 3$, $t\geq 3s-2$ and $(s,t)\neq (3,7)$. If $\big\lfloor\frac{s+t}{2}\big\rfloor-s+1$ is even, then $\frac{3}{2}\big(\big\lfloor\frac{s+t}{2}\big\rfloor-s+1\big)+s\leq t$. If $\big\lfloor\frac{s+t}{2}\big\rfloor-s+1$ is odd, then $\frac{3}{2}\big(\big\lfloor\frac{s+t}{2}\big\rfloor-s+2\big)+s-1\leq t$.}
\vspace{0.05cm}

\begin{proof} We consider four cases.

\medskip\noindent
{\bf Case 1.} $s\geq 3$ is odd and $t\geq 3s-2\geq 7$ is odd.
\vspace{0.05cm}

In this case, we have $\lfloor(s+t)/2\rfloor-s+1=(t-s+2)/2$. If $(t-s+2)/2$ is even, then $3(\lfloor(s+t)/2\rfloor-s+1)/2+s= 3(t-s+2)/4+s= (3t+s+6)/4\leq (3t+(t+2)/3+6)/4= (10t+20)/12\leq t$ unless $(s,t)=(3,9)$. It is easy to check that the lemma holds when $(s,t)=(3,9)$. If $(t-s+2)/2$ is odd, then $3(\lfloor(s+t)/2\rfloor-s+2)/2+s-1=3((t-s+2)/2+1)/2+s-1= (3t+s+8)/4\leq (3t+(t+2)/3+8)/4= (10t+26)/12\leq t$ unless $(s,t)\in\{(3,7), (3,11)\}$. It is easy to check that the lemma holds when $(s,t)=(3,11)$.

\medskip\noindent
{\bf Case 2.} $s\geq 4$ is even and $t\geq 3s-2\geq 10$ is even.
\vspace{0.05cm}

In this case, we have $\lfloor(s+t)/2\rfloor-s+1=(t-s+2)/2$. Similar to Case 1, if $(t-s+2)/2$ is even, then $3(\lfloor(s+t)/2\rfloor-s+1)/2+s\leq (10t+20)/12\leq t$. If $(t-s+2)/2$ is odd, then $3(\lfloor(s+t)/2\rfloor-s+2)/2+s-1\leq (10t+26)/12\leq t$ unless $(s,t)=(4,12)$. It is easy to check that the lemma holds when $(s,t)=(4,12)$.

\medskip\noindent
{\bf Case 3.} $s\geq 3$ is odd and $t\geq 3s-2$ is even.
\vspace{0.05cm}

In this case, we in fact have $t\geq 3s-1\geq 8$. Moreover, we have $\lfloor(s+t)/2\rfloor-s+1=(t-s+1)/2$. If $(t-s+1)/2$ is even, then $3(\lfloor(s+t)/2\rfloor-s+1)/2+s= 3(t-s+1)/4+s= (3t+s+3)/4\leq (3t+(t+1)/3+3)/4= (10t+10)/12\leq t$. If $(t-s+1)/2$ is odd, then $3(\lfloor(s+t)/2\rfloor-s+2)/2+s-1=3((t-s+1)/2+1)/2+s-1= (3t+s+5)/4\leq (3t+(t+1)/3+5)/4= (10t+16)/12\leq t$.

\medskip\noindent
{\bf Case 4.} $s\geq 4$ is even and $t\geq 3s-2$ is odd.
\vspace{0.05cm}

In this case, we in fact have $t\geq 3s-1\geq 11$. Moreover, we have $\lfloor(s+t)/2\rfloor-s+1=(t-s+1)/2$. Similar to Case 3, if $(t-s+1)/2$ is even, then $3(\lfloor(s+t)/2\rfloor-s+1)/2+s\leq (10t+10)/12\leq t$. If $(t-s+1)/2$ is odd, then $3(\lfloor(s+t)/2\rfloor-s+2)/2+s-1\leq (10t+16)/12\leq t$.
\end{proof}

\section{Threshold for $r(K_{n,n}, K_{s,t}, q)=n^2-c$}
\label{ap:O1}

The following result implies that $q=st-\lf\frac{s}{2}\rf+1$ is the threshold for $r(K_{n,n}, K_{s,t}, q)=n^2-c$.

\begin{theorem}\label{th:UnbalanceO1} Let $s, t$ and $k$ be three integers with $t\geq s\geq 2$ and $0\leq k\leq \lf\frac{s}{2}\rf-1$. For sufficiently large $n$, we have
\begin{itemize}
\item[{\rm (i)}] $r(K_{n,n}, K_{s,t}, st-k)=n^2-k$,
\item[{\rm (ii)}] $r\left(K_{n,n}, K_{s,t}, st-\lf\frac{s}{2}\rf\right)\leq n^2-\lf\frac{n}{2}\rf$,
\item[{\rm (iii)}] $r\left(K_{n,n}, K_{s,t}, st-\lf\frac{s}{2}\rf\right)= n^2-\lf\frac{n}{2}\rf$ for odd $s\geq 7$,
\item[{\rm (iv)}] $r\left(K_{n,n}, K_{s,t}, st-\lf\frac{s}{2}\rf\right)= n^2-\lc\frac{n}{2}\rc$ for even $s\geq 14$.
\end{itemize}
\end{theorem}

\begin{proof} We first prove the lower bounds. Axenovich, F\"{u}redi and Mubayi (see \cite[Theorem~5.1]{AxFM}) proved that
$$r\left(K_{n,n}, K_{s,s}, s^2-\ell\right)=
\left\{
   \begin{aligned}
    &n^2-\ell, & & \mbox{if $\ell\leq\lf\frac{s}{2}\rf-1$},\\
    &n^2-\lf\frac{n}{2}\rf, & & \mbox{if $\ell=\lf\frac{s}{2}\rf$ and $s\geq 7$ is odd},\\
    &n^2-\lc\frac{n}{2}\rc, & & \mbox{if $\ell=\lf\frac{s}{2}\rf$ and $s\geq 14$ is even}.
   \end{aligned}
   \right.$$
Moreover, since $K_{s,s}\subseteq K_{s,t}$, we have $r(K_{n,n}, K_{s,t}, st-\ell)\geq r\left(K_{n,n}, K_{s,s}, s^2-\ell\right)$ for any $0\leq \ell \leq \lfloor s/2\rfloor$. This proves the lower bounds.

We now prove the upper bound in (i). Since $0\leq k\leq \lf s/2\rf-1\leq \lf t/2\rf-1$, we have $r(K_{n,n}, K_{s,t}, st-k)\leq r\left(K_{n,n}, K_{t,t}, t^2-k\right)=n^2-k$.

Next, we prove the upper bounds in (ii), (iii) and the case that $n$ is even in (iv) by construction. Let $G=G(A, B)$ be a copy of $K_{n,n}$, where $A=\{a_1, a_2, \ldots, a_n\}$ and $B=\{b_1, b_2, \ldots, b_n\}$. We color the edges of $G$ such that $c(a_{2i-1}b_{2i-1})=c(a_{2i}b_{2i})=i$ for each $1\leq i\leq \lfloor n/2\rfloor$, and all the other edges are colored by $n^2-2\lfloor n/2\rfloor$ new distinct colors. Note that $G$ is an edge-colored $K_{n,n}$ using exactly $n^2-\lfloor n/2\rfloor$ colors, such that every $K_{s,t}$ receives at least $st-\lfloor s/2\rfloor$ distinct colors.

Finally, we prove the upper bounds in (iv) with odd $n$ by construction. Let $G=G(A, B)$ be a copy of $K_{n,n}$, where $A=\{a_1, a_2, \ldots, a_n\}$ and $B=\{b_1, b_2, \ldots, b_n\}$. We color the edges of $G$ such that $c(a_{2i-1}b_{2i-1})=c(a_{2i}b_{2i})=i$ for each $1\leq i\leq (n-1)/2$, $c(a_1b_n)=c(b_1a_n)=(n+1)/2$, and all the other edges are colored by $n^2-(n+1)$ new distinct colors. Note that $G$ is an edge-colored $K_{n,n}$ using exactly $n^2-\lceil n/2\rceil$ colors, such that every $K_{s,t}$ receives at least $st-\lfloor s/2\rfloor$ distinct colors.
\end{proof}

In the case $t\geq s+1$, we can improve Theorem~\ref{th:UnbalanceO1} (iii) and (iv) as follows.

\begin{theorem}\label{th:UnbalanceO1+} For any integers $t> s\geq 3$ and sufficiently large $n$, we have
\begin{itemize}
\item[{\rm (i)}] $r\left(K_{n,n}, K_{s,t}, st-\lf\frac{s}{2}\rf\right)= n^2-\lf\frac{n}{2}\rf$ for odd $s\geq 3$,
\item[{\rm (ii)}] $r\left(K_{n,n}, K_{s,t}, st-\lf\frac{s}{2}\rf\right)= n^2-\lc\frac{n}{2}\rc$ for even $s\geq 10$.
\end{itemize}
\end{theorem}

\begin{proof} Note that the upper bound constructions in the proof of Theorem~\ref{th:UnbalanceO1} in fact hold for all $s\geq 3$. It suffices to prove the lower bounds. Let $G=G(A, B)$ be an edge-colored $K_{n,n}$ such that every $K_{s,t}$ receives at least $st-\lfloor s/2\rfloor$ distinct colors. For a contradiction, suppose that $|C(G)|<n^2-\lfloor n/2\rfloor$ when $s\geq 3$ is odd, and $|C(G)|<n^2-\lceil n/2\rceil$ when $s\geq 10$ is even. Let $C'=\{i\in C(G)\colon\, \mbox{there exist at least two edges with color $i$ in $G$}\}$. Then $|\{e\in E(G)\colon\, c(e)\in C'\}|>\lfloor n/2\rfloor+1$. For each color $i\in C'$, let $e^i_1, \ldots, e^i_{k_i}$ be all the edges of color $i$.

We construct an auxiliary 4-uniform hypergraph $\mathcal{H}$ with $V(\mathcal{H})=V(G)$ as follows. For each $i\in C'$ and $2\leq j\leq k_i$, we form a hyperedge $E^i_j$ by taking $e^i_1\cup e^i_j$ and adding an arbitrary additional vertex if necessary so that $|E^i_j\cap A|=|E^i_j\cap B|=2$. Note that $|E(\mathcal{H})|=n^2-|C(G)|$.

If for every two distinct hyperedges $E^{i_1}_{j_1}$ and $E^{i_2}_{j_2}$ we have $E^{i_1}_{j_1}\cap E^{i_2}_{j_2}\cap A=\emptyset$, then $|E(\mathcal{H})|\leq \lfloor n/2\rfloor$. So $|C(G)|\geq n^2-\lfloor n/2\rfloor \geq n^2-\lceil n/2\rceil$, a contradiction. Hence, there exist two distinct hyperedges $E^{i_1}_{j_1}$ and $E^{i_2}_{j_2}$ with $E^{i_1}_{j_1}\cap E^{i_2}_{j_2}\cap A\neq\emptyset$, and by symmetry, there exist two distinct hyperedges $E^{i_3}_{j_3}$ and $E^{i_4}_{j_4}$ with $E^{i_3}_{j_3}\cap E^{i_4}_{j_4}\cap B\neq\emptyset$.

We first prove the lower bound in (i). Since $E^{i_1}_{j_1}$ and $E^{i_2}_{j_2}$ are two distinct hyperedges with $E^{i_1}_{j_1}\cap E^{i_2}_{j_2}\cap A\neq\emptyset$, the subgraph of $G$ induced by $E^{i_1}_{j_1}\cup E^{i_2}_{j_2}$ is a subgraph of $K_{3,4}$ with at least two color repetitions. This is a contradiction when $s=3$. Thus we may assume that $s\geq 5$. Note that $|\{e\in E(G)\colon\, c(e)\in C'\}|>\lfloor n/2\rfloor+1$. By adding $s-3$ additional edges with at least $(s-3)/2$ color repetitions, we get a subgraph of $K_{s,s+1}$ with at least $2+(s-3)/2=\lfloor s/2\rfloor+1$ color repetitions, a contradiction.

We next prove the lower bound in (ii). Note that the subgraph of $G$ induced by $\bigcup_{\ell=1}^{4}E^{i_{\ell}}_{j_{\ell}}$ is a subgraph of $K_{7,7}$ with at least four color repetitions. If $G\big[\bigcup_{\ell=1}^{4}E^{i_{\ell}}_{j_{\ell}}\big]$ has at least five color repetitions, then by adding $s-7$ additional edges with at least $\lfloor(s-7)/2\rfloor$ color repetitions, we get a subgraph of $K_{s,s}$ with at least $5+\lfloor(s-7)/2\rfloor=\lfloor s/2\rfloor+1$ color repetitions, a contradiction. Thus $G\big[\bigcup_{\ell=1}^{4}E^{i_{\ell}}_{j_{\ell}}\big]$ has exactly four color repetitions. This implies that $\mathcal{H}[\bigcup_{\ell=1}^{4}E^{i_{\ell}}_{j_{\ell}}]$ has exactly four hyperedges. Moreover, if $G\big[\bigcup_{\ell=1}^{4}E^{i_{\ell}}_{j_{\ell}}\big]\subseteq K_{6,7}$, then by adding $s-6$ additional edges with at least $(s-6)/2$ color repetitions, we get a subgraph of $K_{s,s+1}$ with at least $4+(s-6)/2=\lfloor s/2\rfloor+1$ color repetitions, a contradiction. Thus $G[\bigcup_{\ell=1}^{4}E^{i_{\ell}}_{j_{\ell}}]=K_{7,7}$.

If there exists an hyperedge $E^{i_5}_{j_5}\in E(\mathcal{H})\setminus \big\{E^{i_1}_{j_1}, E^{i_2}_{j_2}, E^{i_3}_{j_3}, E^{i_4}_{j_4}\big\}$ such that $E^{i_5}_{j_5}$ and some hyperedge in $\big\{E^{i_1}_{j_1}, E^{i_2}_{j_2}, E^{i_3}_{j_3}, E^{i_4}_{j_4}\big\}$ have a common vertex in $A$, then $G\big[\bigcup_{\ell=1}^{5}E^{i_{\ell}}_{j_{\ell}}\big]$ is a subgraph of $K_{8,9}$ with at least five color repetitions. By adding $s-8$ additional edges to $G\big[\bigcup_{\ell=1}^{5}E^{i_{\ell}}_{j_{\ell}}\big]$, we can get a subgraph of $K_{s,s+1}$ with at least $5+(s-8)/2=\lfloor s/2\rfloor+1$ color repetitions. This is a contradiction. Thus for any $E^{i_5}_{j_5}\in E(\mathcal{H})\setminus \big\{E^{i_1}_{j_1}, E^{i_2}_{j_2}, E^{i_3}_{j_3}, E^{i_4}_{j_4}\big\}$, we have $E^{i_5}_{j_5}\cap \big(\bigcup_{\ell=1}^{4}E^{i_{\ell}}_{j_{\ell}}\big)\cap A=\emptyset$.

If there exist two distinct hyperedges $E^{i_5}_{j_5}, E^{i_6}_{j_6}\in E(\mathcal{H})\setminus \big\{E^{i_1}_{j_1}, E^{i_2}_{j_2}, E^{i_3}_{j_3}, E^{i_4}_{j_4}\big\}$ with $E^{i_5}_{j_5}\cap E^{i_6}_{j_6}\cap A\neq\emptyset$, then $G\big[\bigcup_{\ell=1}^{6}E^{i_{\ell}}_{j_{\ell}}\big]$ is a subgraph of $K_{10,11}$ with at least six color repetitions. This is a contradiction when $s=10$. For $s\geq 12$, by adding $s-10$ additional edges with at least $(s-10)/2$ color repetitions, we get a subgraph of $K_{s,s+1}$ with at least $6+(s-10)/2=\lfloor s/2\rfloor+1$ color repetitions, a contradiction. Therefore, for any two distinct hyperedges $E^{i_5}_{j_5}, E^{i_6}_{j_6}\in E(\mathcal{H})\setminus \big\{E^{i_1}_{j_1}, E^{i_2}_{j_2}, E^{i_3}_{j_3}, E^{i_4}_{j_4}\big\}$, we have $E^{i_5}_{j_5}\cap E^{i_6}_{j_6}\cap A=\emptyset$. Then $|E(\mathcal{H})|\leq 4+\lfloor(n-7)/2\rfloor=\lceil n/2\rceil$. So $|C(G)|\geq n^2-\lceil n/2\rceil$, and this completes the proof.
\end{proof}

\begin{remark}\label{re:UnbalanceO1+} {\rm In Theorem~\ref{th:UnbalanceO1+} (ii), the lower bound $s\geq 10$ on $s$ is sharp. In fact, the following construction shows that $r(K_{n,n}, K_{8,9}, 68)\leq n^2-4\lfloor n/7\rfloor$. Let $G=G(A, B)$ be a copy of $K_{n,n}$, where $A=\{a_1, a_2, \ldots, a_n\}$ and $B=\{b_1, b_2, \ldots, b_n\}$. We color the edges of $G$ such that for each $1\leq i\leq \lfloor n/7\rfloor$, we have $c(a_{7(i-1)+1}b_{7(i-1)+1})=c(a_{7(i-1)+2}b_{7(i-1)+3})=4(i-1)+1$, $c(a_{7(i-1)+1}b_{7(i-1)+2})=c(a_{7(i-1)+3}b_{7(i-1)+4})=4(i-1)+2$, $c(a_{7(i-1)+4}b_{7(i-1)+5})=c(a_{7(i-1)+6}b_{7(i-1)+6})=4(i-1)+3$ and $c(a_{7(i-1)+5}b_{7(i-1)+5})=c(a_{7(i-1)+7}b_{7(i-1)+7})=4(i-1)+4$, and all the other edges are colored by $n^2-8\lfloor n/7\rfloor$ new distinct colors.
}
\end{remark}

\section{Threshold for $r(K_{n,n}, K_{s,t}, q)=n^2-O(n)$}
\label{ap:On}

The following result implies that the threshold for $r(K_{n,n}, K_{s,t}, q)=n^2-O(n)$ is between $st-\lf\frac{s+t-1}{3}\rf+1$ and $st-\lf\frac{2s-1}{3}\rf+1$ when $t\geq s\geq 2$ and $s+t\geq 8$. Moreover, if $t\geq 2(s-1)$ and $s\geq 4$, then the threshold is between $st-\lf\frac{s+t-1}{3}\rf+1$ and $st-s+2$. In particular, if $t= 2(s-1)$ and $s\geq 4$, then $q=st-\lf\frac{s+t-1}{3}\rf+1$ is the threshold for $r(K_{n,n}, K_{s,t}, q)=n^2-O(n)$.

\begin{theorem}\label{th:UnbalanceOn} For any integers $t\geq s\geq 2$ and sufficiently large $n$, the following statements hold.
\begin{itemize}
\item[{\rm (i)}] $r\left(K_{n,n}, K_{s,t}, st-\lf\frac{2s-1}{3}\rf+1\right)>n^2-2\lf\frac{s-2}{3}\rf(n-1)$.
\item[{\rm (ii)}] If $s\geq 3$ and $t\geq 2(s-1)$, then $r\left(K_{n,n}, K_{s,t}, st-s+2\right)\geq n^2-(s-2)n+1$.
\item[{\rm (iii)}] If $s+t\geq 8$, then $r\left(K_{n,n}, K_{s,t}, st-\lf\frac{s+t-1}{3}\rf\right)<n^2-\Theta(n^{1+3/(s+t-3)})$.
\end{itemize}

\end{theorem}

\begin{proof} (i) Since $K_{s,s}\subseteq K_{s,t}$, we have $r(K_{n,n}, K_{s,t}, st-\lfloor(2s-1)/3\rfloor+1)\geq r(K_{n,n}, K_{s,s},$ $s^2-\lfloor(2s-1)/3\rfloor+1)>n^2-2\lfloor(s-2)/3\rfloor(n-1)$ (see \cite[Theorem~6.1]{AxFM}).

(ii) Let $G$ be an edge-colored $K_{n,n}$ such that every $K_{s,t}$ receives at least $st-s+2$ distinct colors. Let $C'=\{i\in C(G)\colon\, \mbox{there exist at least two edges with color $i$ in $G$}\}$. We consider the spanning subgraph $G'$ of $G$ with $E(G')=\{e\in E(G)\colon\, c(e)\in C'\}$. Suppose $|E(G')|> (s-2)n$. Then $G'$ contains a copy $S$ of $K_{1,s-1}$. Let $E'=\{e\in E(S)\colon\, c(e)=c(f) \mbox{\ for some $f\in E(S)\setminus\{e\}$}\}$ and $E''=E(S)\setminus E'$. Then there exist $|E''|$ edges $e_1, \ldots, e_{|E''|}\in E(G)\setminus E(S)$ such that $\{c(e_i)\colon\, 1\leq i\leq |E''|\}=\{c(e)\colon\, e\in E''\}$. Moreover, $S$ has at least $\lceil |E'|/2\rceil$ color repetitions. Note that $s-1+|E''|+|E'|=2(s-1)\leq t$, $1+|E''|+|E'|=s$ and $\lceil |E'|/2\rceil+|E''|+\lfloor |E'|/2\rfloor=s-1$. Then $E(S)\cup \{e_1, \ldots, e_{|E''|}\}$ together with an additional $K_{|E'|, |E'|}$ with at least $\lfloor |E'|/2\rfloor$ color repetitions (this is possible since $|E(G')|> (s-2)n$ and $n$ is large enough) forms a subgraph of $K_{s,t}$ with at least $s-1$ color repetitions, a contradiction. Thus $|E(G')\leq (s-2)n$, so $|C(G)|\geq n^2-(s-2)n+1$.

(iii) Let $\ell=\lfloor(s+t-1)/3\rfloor$. Brown, Erd\H{o}s and S\'{o}s (see \cite[Section~4]{BrES}) proved that there exists a 4-uniform hypergraph $\mathcal{H}$ on $2n$ vertices with $cn^{4-(s+t-4)/\ell}$ hyperedges in which every subset of $s+t$ vertices spans at most $\ell$ hyperedges, where $c>0$ is independent of $n$. Note that $n^{1+3/(s+t-3)}\leq n^{4-(s+t-4)/\ell}\leq n^{1+9/(s+t-1)}$. For each pair of vertices $(u,v)$ in $\mathcal{H}$, let $m(u, v)$ be the number of hyperedges containing both $u$ and $v$ in $\mathcal{H}$. We claim that for any $(u,v)$, we have $m(u, v)\leq \ell$. Otherwise, since $2+2(\ell+1)= 2(\lfloor(s+t-1)/3\rfloor+2)\leq s+t$, there exist $s+t$ vertices spanning at least $\ell+1$ hyperedges in $\mathcal{H}$, a contradiction.

Let $M=\big\{(u, v)\in{V(\mathcal{H}) \choose 2}\colon\, m(u,v)\geq 1\big\}$, and $x=|M|$. Since $1\leq m(u, v)\leq \ell$ for any $(u, v)\in M$, we have $x=\Theta(|E(\mathcal{H})|)$. Among all the spanning sub-hypergraphs of $\mathcal{H}$ satisfying that every pair of vertices are contained in at most one hyperedges, we choose one with maximum number of edges, denoted by $\mathcal{H}'$. For each pair of vertices $(u,v)$ in $\mathcal{H}'$, let $m'(u, v)$ be the number of hyperedges containing both $u$ and $v$ in $\mathcal{H}'$. Let $M'=\big\{(u, v)\in{V(\mathcal{H}) \choose 2}\colon\, m'(u,v)= 1\big\}$, and $y=|M'|$. Note that $M'\subseteq M$. Let $z=|M\setminus M'|=x-y$. We claim $z\leq 5(\ell-1)y$. Otherwise, suppose $z\geq (\ell-1)y+1$. Note that for any $(u, v)\in M'$, we have $1= m'(u, v)\leq m(u, v)\leq \ell$, and for any $(u, v)\in M\setminus M'$, we have $0= m'(u, v)< m(u, v)\leq \ell$. Thus in $\mathcal{H}$, there are at most $(\ell-1)y$ hyperedges containing one pair of vertices in $M'$ and one pair of vertices in $M\setminus M'$. Thus in $\mathcal{H}$, there exists a hyperedge containing two pairs of vertices in $M\setminus M'$ but not pair of vertices in~$M'$. This contradicts the choice of $\mathcal{H}'$. Thus $z\leq 5(\ell-1)y$, so $y\geq x/\ell$. Hence, $|E(\mathcal{H}')|=\Theta(y)=\Theta(x)=\Theta(|E(\mathcal{H})|)$.

We randomly partition $V(\mathcal{H}')$ into two parts $A$ and $B$ with $|A|=|B|=n$. Let $e(A,B)$ denote the number of hyperedges in $\mathcal{H}'$ containing two vertices in $A$ and two vertices in $B$. For any hyperedge $e \in E(\mathcal{H}')$, let $X$ denote the event that $|e\cap A|=|e\cap B|=2$. Then the probability that $X$ appears is ${4\choose 2}/2^4=3/8$. Thus the expectation of $e(A,B)$ is $3|E(\mathcal{H}')|/8$. Hence, there exists a partition $(A,B)$ of $V(\mathcal{H}')$ with $|A|=|B|=n$ such that the number of hyperedges in $\mathcal{H}'$ containing two vertices in $A$ and two vertices in $B$ is $\Theta(|E(\mathcal{H}')|)$. Let $\mathcal{H}''$ be the spanning sub-hypergraph of $\mathcal{H}'$ consisting of these $\Theta(|E(\mathcal{H}')|)$ hyperedges. Denoted by $E(\mathcal{H}'')=\{e_1, \ldots, e_{|E(\mathcal{H}'')|}\}$, where for each $i\in [|E(\mathcal{H}'')|]$,  we have $e_i=\{a^i_1, a^i_2, b^i_1, b^i_2\}$, $a^i_1, a^i_2\in A$ and $b^i_1, b^i_2\in B$.

We form an edge-colored copy $G$ of $K_{n,n}$ with bipartition $(A, B)$ as follows. For any $i\in [|E(\mathcal{H}'')|]$, we color the edges $a^i_1b^i_1$ and $a^i_2b^i_2$ using color $i$. We color all the other edges using $n^2-2|E(\mathcal{H}'')|$ new distinct colors. Since every subset of $s+t$ vertices spans at most $\ell$ hyperedges in $\mathcal{H}''$, every $K_{s,t}$ receives at least $st-\ell$ colors in $G$. Moreover, $|C(G)|=n^2-|E(\mathcal{H}'')|= n^2-\Theta(n^{4-(s+t-4)/\ell})\leq n^2-\Theta(n^{1+3/(s+t-3)})$. This completes the proof.
\end{proof}

\end{document}